\documentclass[a4paper,12pt,reqno]{amsart}

\usepackage[margin=2.5cm]{geometry}

\numberwithin{equation}{section}

\usepackage[T1]{fontenc}
\usepackage[english]{babel}
\usepackage{amsthm}
\usepackage{amsfonts}
\usepackage{amsmath}
\usepackage{amssymb}

\usepackage{mathtools}
\usepackage{enumerate}
\usepackage{hyperref}

\usepackage{enumitem}

\DeclareMathOperator{\SL}{SL}

\DeclareMathOperator{\Gal}{Gal}
\DeclareMathOperator{\rank}{rank}

\DeclareMathOperator{\vol}{vol}

\DeclareMathOperator{\red}{red}
\DeclareMathOperator{\D}{\mathcal{D}}

\DeclareMathOperator{\Disc}{Disc}
\DeclareMathOperator{\Norm}{Norm}

\newcommand{\Q}{\mathbb{Q}}

\newcommand{\R}{\mathbb{R}}
\newcommand{\PP}{\mathbb{P}}
\newcommand{\Z}{\mathbb{Z}}
\newcommand{\F}{\mathbb{F}}
\renewcommand{\O}{\mathcal{O}}

\newcommand{\p}{\mathfrak{p}}
\newcommand{\n}{\mathfrak{n}}

\newcommand{\leg}[2]{\left(\frac{#1}{#2}\right)}

\newcommand{\OO}{\mathcal{O}}

\newcommand{\cI}{\mathcal{I}}
\newcommand{\cJ}{\mathcal{J}}

\newcommand{\ZZ}{\mathbb{Z}}
\newcommand{\QQ}{\mathbb{Q}}
\newcommand{\ve}{\varepsilon}
\newcommand{\y}{\mathbf{y}}

\newcommand{\0}{\mathbf{0}}

\newcommand{\FF}{\mathbb{F}}

\newcommand{\NN}{\mathbb{N}}

\newcommand{\cW}{\mathcal{W}}

\newcommand{\cP}{\mathcal{P}}

\newcommand{\cF}{\mathcal{F}}

\newcommand{\cU}{\mathcal{U}}

\renewcommand{\c}{\mathbf{c}}
\renewcommand{\v}{\mathbf{v}}
\renewcommand{\u}{\mathbf{u}}

\newcommand{\fp}{\mathfrak{p}}

\renewcommand{\epsilon}{\varepsilon}

\newtheorem{theorem}{Theorem}[section]

\newtheorem{lemma}[theorem]{Lemma}
\newtheorem{prop}[theorem]{Proposition}
\newtheorem{conjecture}[theorem]{Conjecture}

\theoremstyle{definition}
\newtheorem*{ack}{Acknowledgements}
\newtheorem{remark}[theorem]{Remark}
\newtheorem{definition}[theorem]{Definition}

\renewcommand{\leq}{\leqslant}
\renewcommand{\geq}{\geqslant}

\usepackage[usenames,dvipsnames]{color}

\begin{document}

\author{Tim Browning}
\author{Stephanie Chan}
\address{IST Austria\\
Am Campus 1\\
3400 Klosterneuburg\\
Austria}
\email{tdb@ist.ac.at}
\email{stephanie.chan@ist.ac.at}

\title[Quadratic twists of an elliptic curve]{Almost all quadratic twists of an elliptic curve have no integral points}

\date{\today}

\begin{abstract}
For a given elliptic curve $E$ in 
short Weierstrass form, we show that almost all quadratic twists $E_D$ have no integral points, as $D$ ranges over square-free integers ordered by size. Our result is conditional 
on a weak form of the Hall--Lang conjecture in the case that $E$ has partial $2$-torsion. 
The proof uses a correspondence of Mordell and the reduction theory of binary quartic forms in order to transfer the problem to counting rational points of bounded height on a certain singular cubic surface, together with extensive use of cancellation in character sum estimates, drawn from  Heath-Brown's  analysis of Selmer group statistics for the congruent number curve.
\end{abstract}

\subjclass[2010]{11D45 (11D25, 11G05)}

\maketitle

\setcounter{tocdepth}{1}
\tableofcontents

\section{Introduction}

Given an elliptic curve defined over $\ZZ$, 
the aim of this paper is to establish the paucity of quadratic twists that contain integral points. 
Fixing $A$ and $B$ such that $4A^3+27B^2\neq 0$, 
let $E$ denote the elliptic curve 
$y^2=x^3+Ax+B$.  
We shall consider the quadratic twist family 
\begin{equation}\label{eq:jive}
E_D:y^2=x^3+AD^2x+BD^3,
\end{equation}
as $D$ runs over square-free integers.
Let $\D\coloneqq\{D\in\Z \text{ square-free}\}$ and let 
$$
\D(N)\coloneqq\{D\in\D:|D|\leq N\},
$$
for any $N\geq 1$.
Denote the set of integral points on $E_D$ by
\[
E_D(\Z)\coloneqq \{(x,y)\in\Z^2:y^2=x^3+AD^2x+BD^3\}.
\]
If $E$ has a rational $2$-torsion point and $\alpha$ is an integer root of the polynomial $x^3+Ax+B$, then every twist $E_D$ has an integral point $(\alpha D,0)$. 
Accordingly, we define 
\begin{equation}\label{eq:salsa}
E_D^*(\Z)\coloneqq E_D(\Z)\setminus E_D[2],
\end{equation}
where 
$E_D[2]$ is the $2$-torsion subgroup 
of $E_D(\Q)$.
Granville \cite{Granvilletwists} has conjectured asymptotics for the density of elliptic curves 
that have a non-trivial integral point in a quadratic twist family. (However, care should be taken when comparing the conjecture against our setting, since he considers the  model $Dy^2=x^3+Ax+B$, which contains fewer integral points than $E_D$.) In our setting, we expect that 
$$
\#\{D\in\D(N):E_D^*(\Z)\neq \varnothing\}\sim c_{A,B} N^{\frac{1}{2}},
$$ 
as $N\to \infty$, for some constant $c_{A,B}>0$ depending only on $A$ and $B$.
Our main goal is to show that the left hand side is $o(N)$.  
Chan \cite{congruentav} has achieved this for  the quadratic twist family of congruent number curves $y^2=x^3-D^2x$. Moreover, for a fixed square-free integer $k\neq 1$, 
Chan \cite{cubicav} has also proved that almost all elliptic curves in the cubic twist family of Mordell curves $y^2=x^3+kD^2$ have no integral point.

In the  setting of quadratic twists of an arbitrary elliptic curve $y^2=x^3+Ax+B$, 
we shall need to work under the following  hypothesis when the 
 curve has partial $2$-torsion. 

\begin{conjecture}[weak Hall--Lang]
\label{con}
Let $E$ be the elliptic curve $y^2=x^3+Ax+B$, where $A,B\in \ZZ$ are such that $4A^3+27B^2\neq 0$, 
and let  $(x,y)\in E(\ZZ)$. Then there exists constants $C,\ve>0$ such that  
$|x|\leq \exp\left(C\max\{|A|^{\frac{1}{2}},|B|^{\frac{1}{3}}\}^{1-\ve}\right). 
$
\end{conjecture}

The {\em Hall--Lang conjecture} is recorded by  Lang \cite{lang} and predicts that 
a polynomial bound should hold. Thus the conjecture states that 
 there exist constants $C,\kappa>0$ such that 
$$
|x|<C\max\{|A|,|B|\}^\kappa.
$$
Note that the analogue of this  conjecture over $\FF_q(t)$  is known, thanks to  Schmidt \cite{schmidt}. Over $\QQ$, on the other hand, we only have the exponential bound 
$$
|x|\leq \exp\left(C_\ve\max\{|A|^{\frac{1}{2}},|B|^{\frac{1}{3}}\}^{6+\ve}\right),
$$ 
for any $\ve>0$, which follows from work of 
Hajdu and Herendi \cite{Bug}. 

We are now ready to reveal our  main result.

\begin{theorem}\label{t:main}
Let $\ve>0$. 
Let  $A,B\in \ZZ$ such that $4A^3+27B^2\neq 0$, and let $E_D$ be given by \eqref{eq:jive}.
Assume that Conjecture \ref{con} holds.
Then 
\[
N^{\frac{1}{2}}\ll 
\#\{D\in\D(N):E_D^*(\Z)\neq \varnothing\}\ll \frac{N}{(\log N)^{\frac{1}{8}-\ve}},
\]
where the implied constants depend at most on $A, B$ and $\ve$. 
Moreover, Conjecture \ref{con} is only required in the proof of the upper bound, and only then when 
$x^3+Ax+B$ has precisely one root over $\QQ$.
\end{theorem}

When the underlying curve $E_1$ does not have rational two torsion, so that 
$E_D[2]=\varnothing$, not only is our work completely unconditional, but the upper bound 
holds with $\frac{1}{3}$ instead of $\frac{1}{8}$. 
A detailed summary of our results is presented in Section \ref{s:mom},  which we shall build on to address 
moments of $\#E_D^*(\Z)$. For comparison, 
for the full family of elliptic curves, 
Alp\"{o}ge \cite{Alpoge} has shown that the average number of integral points on  elliptic curves ordered by height is bounded (by at most $66$). More recently, Alp\"{o}ge and Ho 
 \cite{AlpogeHo}
have shown that the second moment is also bounded in the family of all elliptic curves. 
Restricting to the setting of quadratic twist families, we will apply  recent work of Smith \cite{Smith2} to 
assess higher moments of  $\#E_D(\Z)$ in the case that 
$E_D[2]=\varnothing$.

\begin{theorem}\label{t:mom}
Let $A,B\in\Z$ such that $x^3+Ax+B$ is irreducible over $\Q$.
Then, for any positive integer $k\leq \log\log\log N$, we have
\[
\frac{1}{\#\D(N)}
\sum_{D\in\D(N)} \#E_D(\Z)^k\ll \frac{1}{(\log N)^{\frac{1}{4}}} ,
\]
 where the implied constant depends at most on $A$ and $B$.
\end{theorem}

Let us proceed by summarising our approach to Theorem \ref{t:main}.
Our argument differs according to the size of $\gcd(x_P,D)=g_P$, say, where $P=(x_P,y_P)\in E_D(\ZZ)$.
When $g_P$ is large we shall reduce to a question about the solubility of $x^3+Ax+B$ modulo $g_P$. When  $x^3+Ax+B$  is irreducible over $\QQ$, 
the Chebotarev density theorem  will allow us to control the 
density of primes $p$ for which $x^3+Ax+B$ has a root modulo $p$.
Alternatively, when $x^3+Ax+B$  is reducible over $\QQ$, we shall 
reduce to a question about the solubility of a system of equations, for which 
 character sum estimates will prove crucial. 

When $g_P$ is small, on the other hand, 
our approach makes crucial use of a construction by Mordell 
\cite[Chapter~25]{Mordell}, which associates a binary quartic form 
$f_P\in \ZZ[x_1,x_2]$
to any point 
$P\in E_D(\Z)$. Building on the 
discriminant-lowering procedure adopted for binary cubic forms in 
\cite{cubicav}, we will use the reduction theory of binary quartic forms, as explained by Cremona \cite{Cremona}, in order to lower the discriminant of $f_P$. 
This will ultimately allow us to reduce the problem to counting rational points of bounded height on a certain singular cubic surface, which may be of independent interest.

\medskip

Let 
$C\in \ZZ[x_1,x_2]$ be a separable binary cubic form and let  $S\subset \PP^3$ be the cubic surface
\begin{equation}\label{eq:ogeq}
C(x_1,x_2)=x_3^2x_4. 
\end{equation}
This is a singular cubic surface containing six lines, and with an isolated singularity 
$(0:0:0:1)$ of type $\mathbf{D}_4$.
The Manin conjecture \cite{fmt} makes a precise prediction for the asymptotic behaviour of the quantity
$$
N(U;B)=\#\left\{x\in U(\QQ): H(x)\leq B\right\},
$$
as $B\to \infty$, where 
$H$ is standard height function on $\PP^3(\QQ)$ and 
$U$ is the Zariski open subset formed by deleting the six lines from $S$. A crude form of the conjecture predicts linear growth, so that $N(U;B)=O_S(B^{1+\ve})$, for any $\ve>0$.

When the cubic form factorises as $C(x_1,x_2)=x_1x_2(x_1+x_2)$, the surface $S$ has been studied by Browning \cite{d4}, who established that
$$
B(\log B)^6\ll N(U;B) \ll B(\log B)^6.
$$
This was later upgraded to an asymptotic formula for $N(U;B)$ by Le Boudec \cite{d4-a}.
For Theorem \ref{theorem:maintwists} we shall need to study $S$ for any separable binary cubic form $C$.  In fact, it will suffice to restrict  to the locus of $(x_1:x_2:x_3:x_4)
 \in S(\QQ)$ for which $\gcd(x_1,x_2,x_4)=1$, as follows.

\begin{theorem}\label{theorem:main}
Let $C\in \ZZ[x_1,x_2]$ be a binary cubic form that is separable over $\Q$.
Let $N^\circ(B)$ be the number of $(x_1,x_2,x_3,x_4)\in\Z^4$ with $\gcd(x_1,x_2,x_4)=1$, $x_3x_4\neq 0$, and $|x_i|\leq B$, such that~\eqref{eq:ogeq} holds.
Then
\[B\ll N^\circ(B)\ll B(\log B)^{\max\{\lambda,2\}},\]
where $\lambda$ is the number of irreducible factors of $C$ over $\Q$, and the implied constant depends at most on $C$.
\end{theorem}

The condition $x_3x_4\neq 0$ places us on the open set $U$, so that $N^\circ(B)\leq N(U;B)$.
The condition $\gcd(x_1,x_2,x_4)=1$ corresponds to  counting integral points on $S\setminus D$, with codimension 2 boundary divisor  $(0:0:1:0)$. This counting problem admits an interpretation 
through work of Chambert-Loir and Tschinkel \cite{ACL}.

Our method of proof is based on viewing the equation~\eqref{eq:ogeq} as a congruence modulo $x_3^2$ and 
relies crucially on the particular coprimality condition that is assumed of the solutions. 
When $C$ is irreducible we have 
 $N(S;B)=N(U;B)+O(1)$, since then  none of the six lines are defined over $\QQ$, and  our work implies that $N(S;B)\gg B$. In this case it would be 
  interesting to have a proof of  the corresponding upper bound 
  $N(S;B)\ll  B$.

\begin{ack}
The authors are grateful to Roger Heath-Brown for useful comments. The first author was
supported by FWF grant P~32428-N35.
\end{ack}

\section{Main results and moments}\label{s:mom}

Let  $A,B\in \ZZ$ such that $4A^3+27B^2\neq 0$, and let $E_D$ be given by \eqref{eq:jive}.
We recall the definition \eqref{eq:salsa} of 
$E_D^*(\Z)$.
In this section we summarise our various results, which differ according to the factorisation properties of
$x^3+Ax+B$. The following three results will be deduced in Sections \ref{s:5.4},
\ref{s:7} and \ref{s:8}, respectively. 

\begin{theorem}\label{theorem:maintwists}
Assume that $x^3+Ax+B$ is irreducible over $\Q$.
Then
\[N^{\frac{1}{2}}\ll \#\{D\in\D(N): E_D(\Z)\neq\varnothing\}\ll N(\log N)^{-\frac{1}{3}}\log\log N,\]
where the implied constants depend at most on $A$ and $B$.
\end{theorem}

\begin{theorem}\label{theorem:mainfullt}
Let $\ve>0$. 
Assume that $x^3+Ax+B$ has three distinct roots over $\Q$.
Then
\[
N^{\frac{1}{2}}\ll \#\{D\in\D(N):E_D^*(\Z)\neq \varnothing\}
\ll 
N(\log N)^{-\frac{1}{8}+\epsilon},
\]
where the implied constants depend at most on $A, B$ and $\epsilon$.
\end{theorem}

\begin{theorem}\label{theorem:mainpartt}
Assume that $x^3+Ax+B$ factors as $(x-r)Q(x)$,  for some $r\in \Q$ and a polynomial $Q$ that is irreducible over $\Q$. 
Assume that either 
$Q(r)<0$, or $Q(r)\in\Q^2$, or else that 
Conjecture \ref{con} holds.
Then
\[
N^{\frac{1}{2}}\ll 
\#\left\{D\in\D(N): E_D^*(\Z)\neq \varnothing\right\}
\ll N(\log N)^{-\frac{1}{8}}\log\log N,
\]
where the implied constants depend at most on $A$ and $B$.
\end{theorem}

Assume that  $Q(r)>0$ and $Q(r)\notin\Q^2$. 
Then, as we shall see in Section~\ref{s:8}, 
we shall also be able to give an unconditional proof of the upper bound 
\[
\#\left\{D\in\D(N):\begin{array}{l}
x\cdot Q(r)\notin\Q^2 \text{ or }
\gcd(x,D)< N(\log N)^{-\frac{49}{4}}\\
\text{for some }(x,y)\in E_D(\Z)
\end{array}
\right\}
\ll N(\log N)^{-\frac{1}{8}}\log\log N,
\]
where the implied constant depends at most on $A$ and $B$.

\medskip

Turning to moments,  Theorem \ref{t:mom} is a straightforward consequence of the following general statement about  moments of $\#E^*_D(\Z)$. 

\begin{theorem}\label{theorem:tomoments}
Suppose that 
$$
\#\{D\in\D(N):E^*_D(\Z)\neq \varnothing\}\ll N(\log N)^{-\kappa},
$$ 
for some $\kappa>0$, where the implied constant depends at most on $A,B$ and $\kappa$. 
Then for any positive integer $k\leq \log\log\log N$ and any $\ve>0$, we have
\[
\frac{1}{\#\D(N)}
\sum_{D\in\D(N)} \#E^*_D(\Z)^k\ll (\log N)^{-\kappa+\ve} ,
\]
 where the implied constants depend at most on $A,B,\kappa$ and $\ve$.
\end{theorem}

We obtain Theorem \ref{t:mom} by applying  Theorem \ref{theorem:maintwists} in Theorem \ref{theorem:tomoments}, with $\kappa=\frac{1}{3}-\frac{1}{24}$ and $\ve=\frac{1}{24}$.
Theorem~\ref{theorem:tomoments} will be a consequence of H\"{o}lder's inequality and the following two results.
The first is a special case of recent work by Smith \cite[Theorem~1.1]{Smith2}.

\begin{theorem}[Smith]\label{theorem:Smith}
Let $A,B$ be integers such that $4A^3+27B^2\neq 0$, and let
$E_D$ be given by \eqref{eq:jive}. Then there are real numbers $C,C'>0$, depending only on $A,B$, such that for any $m>0$ and  $N>C$ satisfying $m<C'\log\log\log N$, we have
\[
\frac{1}{\#\D(N)}
\sum_{D\in\D(N)}\exp(m\rank E_D(\Q))\leq \exp(Cm^2).
\]
\end{theorem}

The second ingredient we require is a straightforward  consequence of a result by Hindry and Silverman \cite[Theorem~0.7]{HS}.

\begin{theorem}\label{theorem:HS}
Let $A,B$ be integers such that $4A^3+27B^2\neq 0$, 
and let
$E_D$ be given by \eqref{eq:jive} for $D\in \mathcal{D}$.
Then there exists an absolute constant $c$ such that 
\[
\# E_D(\Z)\leq c^{\omega(\gcd(A,B))+(1+\rank E_D(\Q))\sigma_{E_D}},
\]
where 
\[
\sigma_{E_D}=\frac{\log |\text{discriminant of }E_D|}{\log |\text{conductor of }E_D|}
\]
is the Szpiro ratio of $E/\Q$.
\end{theorem}
\begin{proof}
If $y^2=x^3+AD^2x+BD^3$ gives a quasi-minimal model, so that any $p^4\mid AD^2$ implies that $p^6\nmid BD^3$, then the claim is immediate from \cite[Theorem~0.7]{HS}. Otherwise, the model $y^2=x^3+AD^2x+BD^3$ is not quasi-minimal. Since  $D\in\D$ is assumed to be square-free, any prime that satisfies both $p^4\mid AD^2$ and $p^6\mid BD^3$ must divide $A$ and $B$. Let $S$ be the set of primes dividing both $A$ and $B$. Suppose that $C$ is the largest integer satisfying both $C^4\mid AD^2$ and $C^6\mid BD^3$, then any prime dividing $C$ must be in $S$. If $(x,y)\in\Z^2$ satisfies $y^2=x^3+AD^2x+BD^3$, then $(\frac{x}{C^2},\frac{y}{C^3})$ gives an $S$-integral point on the quasi-minimal model 
$$
y^2=x^3+\frac{AD^2}{C^4}x+\frac{BD^3}{C^6}.
$$ 
Therefore the result follows from \cite[Theorem~0.7]{HS}.
\end{proof}

\begin{proof}[Proof of Theorem~\ref{theorem:tomoments}]
The discriminant of $E_D$ in \eqref{eq:jive}
 is $-16(4A^3+27B^2)D^6$ and the conductor is divisible by primes dividing the discriminant.
Within a quadratic twist family, since $D\in\D$ is a square-free, we deduce that the Szpiro ratio 
is 
\[\sigma_{E_D}\leq 6+\frac{2\log 16|4A^3+27B^2|}{\log |D|}.\]
Thus $\sigma=\sup_{D\in \D}\{\sigma_{E_D}\}\ll 1$, for an implied constant that only depends on  $A$ and $B$.

We shall apply H\"{o}lder's inequality with $p,q>1$ such that $\frac{1}{p}+\frac{1}{q}=1$. This yields
\[
\sum_{D\in\D(N)} \#E^*_D(\Z)^k\leq \#\{D\in\D(N): E^*_D(\Z)\neq \varnothing\}^{\frac{1}{p}}\left(\sum_{D\in\D(N)} \#E_D(\Z)^{qk}\right)^{\frac{1}{q}}.
\]
Applying 
Theorem~\ref{theorem:HS} and the assumption of the theorem, it follows that 
 \[
 \sum_{D\in\D(N)} \#E^*_D(\Z)^k\ll \left(\frac{N}{(\log N)^{\kappa}}\right)^{\frac{1}{p}} c_1^{k}\left(\sum_{D\in\D(N)} (c^{qk\sigma})^{\rank E_D(\Q)}\right)^{\frac{1}{q}},
 \]
where $c_1=c^{\omega(\gcd(A,B))+\sigma}$
and the implied constant depends on $A$ and $B$.
We now apply Theorem~\ref{theorem:Smith} with $m=qk\sigma\log c$. 
Since $k\leq \log\log\log N$, by assumption, we have 
$m\ll \log\log\log N$ if we assume that $q\ll 1$. This  gives 
\[
\sum_{D\in\D(N)} \#E^*_D(\Z)^k\ll \frac{N c_1^{k}c_2^{qk^2}}{(\log N)^{\frac{\kappa}{p}}} ,
\]
where $c_2=\exp(C(\sigma\log c)^2)$.
We now specify the choice of $p,q$ by imposing $q=\frac{1}{\epsilon}$. 
Since $k\leq \log\log\log N$, this leads to 
\begin{align*}
\sum_{D\in\D(N)} \#E^*_D(\Z)^k
&\ll \frac{N}{(\log N)^{(1-\epsilon)\kappa}} \cdot \exp\left(
C_\ve (\log\log\log N)^2\right),
\end{align*}
for a suitable constant $C_\ve$ depending on $A,B$ and $\ve$.
The statement of the theorem easily follows, on redefining $\ve$.
\end{proof}

\section{Counting points on a cubic surface}
\subsection{Preliminaries}

We begin by examining some properties of a binary cubic form $C$, which we  assume to be separable over $\QQ$ and such that $C(1,0)\neq 0$. Denote the discriminant of $C$ by $\Delta(C)$.
Henceforth we allow all implied constants to depend on the coefficients of $C$.
 Since $C$
is separable over $\QQ$,
there exist binary linear forms
 $L_1,L_2,L_3\in \overline{\Q}[y_1,y_2]$
 such that $C=L_1L_2L_3$, with no two factors proportional.
We have the following simple result. 

\begin{lemma}\label{lem:linear}
Assume that $\y=(y_1,y_2)\in \ZZ^2$ such that 
$
|L_1(\y)|\leq |L_2(\y)|\leq |L_3(\y)|.
$
Then $L_2(\y)\gg |\y|$ and $L_3(\y)\ll |\y|$.
\end{lemma}

\begin{proof}
The upper bound is trivial. To see the lower bound, on renormalising and possibly 
interchanging the roles of $y_1,y_2$, we may assume without loss of generality that 
$y_1=1$ and $y_2=t$ for $t\in [-1,1]$.
Let $l_i(t)=L_i(1,t)$ for $1\leq i\leq 3$. Thus there exist 
non-zero $\alpha_i,\beta_i\in \overline{\Q}$ such that 
$$
l_i(t)=\alpha_i+\beta_it,
$$
with $\alpha_i\beta_j-\alpha_j\beta_i\neq 0$. We wish to prove that $l_2(t)\gg 1$, where the implied constant is allowed to depend on the constants 
$\alpha_i,\beta_i$. We suppose for a contradiction that, for any positive integer $n$, there exists 
$t_n\in [-1,1]$ such that $|l_2(t_n)|<1/n$. Then it follows that 
$$
\beta_1 l_2(t_n)=\beta_1\alpha_2-\alpha_1\beta_2+\beta_2l_1(t_n).
$$
Since $|l_1(t_n)|\leq |l_2(t_n)|<1/n$, this therefore implies that 
$$
0<|\beta_1\alpha_2-\alpha_1\beta_2|\ll \frac{1}{n},
$$
which is a contradiction for sufficiently large $n$.
\end{proof}

Our remaining results in this section concern the solubility of $C(x,1)$ and $C(1,x)$ in residue classes.
Define a multiplicative function 
\begin{equation}\label{eq:defrho}
\varrho(n)\coloneqq \#\{x\in \Z/n\Z:C(x,1)\equiv 0\bmod n\}.\end{equation}
It follows from Chebotarev density theorem and Burnside's lemma (see for example \cite[Proposition~3.10 and Section~3.3.3.5]{SerreNp}) that
\begin{equation}\label{eq:Cheb}
 \sum_{p\leq N}\frac{\varrho(p)}{p}= \lambda\log\log N+O(1),
\end{equation}
where $\lambda$ is the number of irreducible factors of $C(x,1)$ over $\Q$.
It follows from Hensel lifting that 
\begin{equation}\label{eq:hensel}
\varrho(p)=\varrho(p^v), \text{ for any $v\geq 1$ and any prime $p\nmid \Delta(C)$.}
\end{equation}
We may also define
\begin{equation}\label{eq:defrho'}
\varrho'(n)\coloneqq \#\{x\in \Z/n\Z:C(1,x)\equiv 0\bmod n\}.
\end{equation}
It follows from work of 
 Huxley \cite{Hux} that 
 \begin{equation}\label{eq:hux}
\varrho(p^v),\varrho'(p^v)\leq 3p^{\frac{1}{2}v_p(\Delta(C))},
\end{equation}
for any prime power $p^v$.
We proceed by proving the following result. 

\begin{lemma}\label{lemma:meanf}
Let $f$ 
be any multiplicative function satisfying
\[
f(p^v)
\begin{cases}
= \varrho(p^v)&\text{if } p\nmid C(1,0)\Delta(C),\\ 
\in [0, \varrho(p^v)+\varrho'(p^v)]&\text{if } p\mid C(1,0)\Delta(C).
\end{cases}
\]
Then we have
\begin{equation}\label{eq:meanf}\sum_{n\leq N} f(n)\asymp N(\log N)^{\lambda-1},\qquad \sum_{n\leq N} \frac{f(n)}{n}\asymp (\log N)^{\lambda}\quad \text{ and }\quad \sum_{n\leq N} \frac{f(n)}{n^2}\asymp1,
\end{equation}
where the implied constants depend only on $C$.
\end{lemma}

\begin{proof}
It follows from~\eqref{eq:hux} that 
 $f(p^v)\leq 6 p^{\frac{1}{2}v_p(\Delta(C))}=O(1)$.
Therefore, by \cite[Satz~1]{Wirsing}, we have
\begin{equation}\label{eq:fntofp}
\sum_{n\leq N}f(n)\asymp \frac{N}{\log N}\exp\left(\sum_{p\leq N}\frac{f(p)}{p}\right).
\end{equation}
The value of $f(p)$ is always equal to $\varrho(p)$ unless $p\mid C(1,0)$, so 
\[\sum_{p\leq N}\frac{f(p)}{p}= \sum_{p\leq N}\frac{\varrho(p)}{p}+O\left(\sum_{p\mid C(1,0)\Delta(C)}\frac{\varrho(p)+\varrho'(p)}{p}\right).\]
Since $ C(1,0)\neq 0$, we deduce that the sum over $p\mid C(1,0)\Delta(C)$ has only finitely many terms. We now apply the estimate from~\eqref{eq:Cheb}, together with the trivial bound $\varrho(p),\varrho'(p)\leq p$ for all $p\mid C(1,0)\Delta(C)$. Thus it follows that
\begin{equation}\label{eq:fpmean}
\sum_{p\leq N}\frac{f(p)}{p}=\lambda\log\log N+O(1).
\end{equation}
The first bound in~\eqref{eq:meanf} follows from combining~\eqref{eq:fntofp} and~\eqref{eq:fpmean}.
The second and third bounds in~\eqref{eq:meanf} can be deduced by partial summation.
\end{proof}

\subsection{Proof of the upper bound in Theorem~\ref{theorem:main}}

Recall our convention that all implied constants are allowed to depend on the coefficients of $C$.
We begin by showing that it suffices to assume without loss of generality that $C(1,0)\neq 0$ in the proof. To see this, we first choose  $a\in \ZZ$, with $a=O(1)$,  such that 
$C(1,a)\neq 0$. But then, on making the
change of variables $(x_1,x_2)\mapsto (x_1,x_2+ax_1)$, we obtain 
a new counting problem, in which $B$ is replaced by $2\max\{1,|a|\}B\ll B$  and the relevant 
binary cubic form has non-zero leading coefficient.

Consider a solution $(x_1,x_2,x_3,x_4)\in\Z^4$ to~\eqref{eq:ogeq}. Write $h_1^2h_2=\gcd(x_1,x_2)$, where $h_1,h_2$ are positive integers and $h_2$ is square-free.
The assumption that $\gcd(x_1,x_2,x_4)=1$ implies that $(h_1^2h_2)^3\mid x_3^2$, so $h_1^3h_2^2\mid x_3$.
Define
\[y_1=\frac{x_1}{h_1^2h_2},\qquad 
y_2=\frac{x_2}{h_1^2h_2},\qquad 
u=\frac{x_3}{h_1^3h_2^2},\qquad
v=x_4.
\]
Then~\eqref{eq:ogeq} can be rewritten as
\begin{equation}\label{eq:redeq}
 C(y_1,y_2)=h_2u^2v.
\end{equation}

First fix $h_1,h_2$ and define the quantity
\begin{equation}\label{eq:flat}
N(Y,U,V)\coloneqq\#\left\{(y_1,y_2,u,v)\in\Z^4: \begin{array}{l}
Y\leq \max\{|y_1|,|y_2|\}< 2Y\\ 
U\leq |u|< 2U,~V\leq |v|< 2V\\ 
\gcd(y_1,y_2)=1,~\eqref{eq:redeq}\text{ holds}
\end{array}\right\},
\end{equation}
with $Y, U, V$ in the range
\begin{equation}\label{eq:rangeYUV}
Y< \frac{B}{h_1^2h_2},\qquad
U< \frac{B}{h_1^3h_2^2},\qquad
V<B.\end{equation}
We shall estimate 
$N(Y,U,V)$ using the geometry of numbers, 
by first proving that the equation 
\eqref{eq:redeq} forces the solutions to lie on a small number of lattices. 

\begin{lemma} \label{lemma:latticebd}
Let $\varrho$ and $\varrho'$ be given by~\eqref{eq:defrho} and~\eqref{eq:defrho'}, respectively.
Define a multiplicative function $f$ via
\begin{equation}\label{eq:deff}
f(p^v)=\begin{cases}
 \varrho(p^v)&\text{if } p\nmid C(1,0),\\ 
 \varrho'(p^v)&\text{if } p\nmid C(0,1),\\
 \varrho(p^v)+\varrho'(p^v)&\text{if } p\mid \gcd(C(1,0),C(0,1)).\\
\end{cases}\end{equation}
Then the solutions $(y_1,y_2)\in\Z^2$ to
$$
 C(y_1,y_2)\equiv 0\bmod d.
$$
 such that $\gcd(y_1,y_2)=1$
 are covered by at most $f(d)$ many rank $2$ lattices, each of determinant $d$.
\end{lemma}
\begin{proof}
By the Chinese remainder theorem, it suffices to consider the case when $d=p^v$.
Since $y_1$ and $y_2$ are coprime, $p$ divides at most one of $y_1$ and $y_2$. Suppose $p\nmid y_2$, then we can rewrite the congruence
 $C(y_1,y_2)\equiv 0\bmod p^v$ as
\[C(y_1/y_2,1)\equiv 0\bmod p^v.\] 
Therefore the solutions $(y_1,y_2)$ all lie in a lattice generated by $(p^v,0)$ and $(\alpha,1)$, where $\alpha\in\Z$ is a solution to $C(x,1)\equiv 0\bmod p^v$.
Moreover, there are $\varrho(p^v)$ many choices of $\alpha$ modulo $p^v$.

If $p\mid y_2$, then since $p\nmid y_1$, we must have $p\mid C(1,0)$. In this case, consider $C(1,y_2/y_1)\equiv 0\bmod p^v$. Then the solutions $(y_1,y_2)$ lies in a lattice generated by $(0,p^v)$ and $(1,\beta)$, where $\beta\in\Z$ is a solution to $C(1,\beta)\equiv 0\bmod p^v$. 

Therefore, the number of lattices is bounded by $\varrho(p^v)$ if $p\nmid C(1,0)$ and $\varrho(p^v)+\varrho'(p^v)$ if $p\mid C(1,0)$. By symmetry we obtain the statement in the lemma.
\end{proof}

Using this result we can now proved the following bound. 

\begin{lemma}\label{lemma:ellipse}
Fix non-zero $h_2,u\in \ZZ$. Then 
 \[
 \#\left\{(y_1,y_2,v)\in\Z^4: \begin{array}{l}
Y\leq \max\{|y_1|,|y_2|\}< 2Y\\ 
\gcd(y_1,y_2)=1, ~V\leq |v|< 2V\\ 
\eqref{eq:redeq}\text{ holds}
 \end{array}
 \right\}\ll f(h_2u^2)\left(\frac{V}{Y}+1\right),\]
where $f$ is the multiplicative function defined in~\eqref{eq:deff} and the implied constant depends only on $C$.
\end{lemma}

\begin{proof}
We begin by applying Lemma~\ref{lemma:latticebd} to $d=h_2u^2$, and proceed to bound the number of solutions $(y_1,y_2)$ that lie in one of the $f(d)$ lattices of determinant $d$.
Note that~\eqref{eq:redeq} has no solution unless 
$h_2\ll Y^3$ and $U\ll Y^{3/2}$.
 By symmetry, on multiplying by $3$, we may assume that 
 $$
 |L_1(y_1,y_2)|\leq |L_{2}(y_1,y_2)|\leq |L_{3}(y_1,y_2)|.
 $$
It follows from Lemma \ref{lem:linear} that 
any such solution must satisfy 
\begin{equation}\label{eq:region}
L_1(y_1,y_2)\ll \frac{h_2U^2V}{Y^2}.
\end{equation}
The region in $\R^2$ cut out by 
$\max\{|y_1|,|y_2|\}<2Y$
 and~\eqref{eq:region} can be placed inside an ellipse 
 $\mathcal{R}\subset \R^2$ 
with area
$$
\vol(\mathcal{R})
\ll Y\cdot \frac{h_2U^2V}{Y^2}= \frac{h_2U^2V}{Y}.
$$
By \cite[Lemma~2]{HBlattice}, there are 
\[\ll \frac{h_2U^2V}{dY}+1\ll \frac{V}{Y}+1\]
many lattice points in the ellipse such that $\gcd(y_1,y_2)=1$.
\end{proof}

\begin{lemma}\label{lemma:cubicub}
In the notation of Theorem~\ref{theorem:main}, we have
$N^\circ(B)\ll B(\log B)^{\max\{\lambda,2\}}$.
\end{lemma}
\begin{proof}
Recall the definition~\eqref{eq:flat} of $N(Y,U,V)$. We apply Lemma~\ref{lemma:ellipse} and sum over $U\leq |u|<2U$. 
Lemma~\ref{lemma:meanf} implies that 
$$
\sum_{U\leq |u|<2U}f(|u|)\ll U.
$$
Now it follows from combining the Chinese remainder theorem with~\eqref{eq:hensel} and~\eqref{eq:hux}
that $f(h_2u^2)\ll f(h_2)f(|u|)$. Hence we deduce that
\[N(Y,U,V)\ll \sum_{U\leq |u|<2U}f(h_2u^2)\left(\frac{V}{Y}+1\right)
\ll f(h_2)\cdot U\cdot\left(\frac{V}{Y}+1\right).
\]
Note that $N(Y,U,V)=0$ unless $h_2U^2V\ll Y^3$, so we can impose $U\ll (Y^3/(h_2V))^{1/2}$.
Now summing $N(Y,U,V)$ over $U,V,Y$ being powers of $2$ subject to~\eqref{eq:rangeYUV}, we have
\[\sum_{U,V,Y}N(Y,U,V)
\ll f(h_2)
\sum_{Y,V}\left(\frac{V^{1/2}Y^{1/2}}{h_2^{1/2}}+\frac{B}{h_1^3h_2^2}\right)\ll \frac{f(h_2)B}{h_1h_2}+\frac{f(h_2)B}{h_1^3h_2^2}(\log B)^2.\]
Finally, we sum over $h_1,h_2$ and apply Lemma~\ref{lemma:meanf} to bound 
$$
\sum_{h_2\leq B}\frac{f(h_2)}{h_2} \quad \text{ and } \quad \sum_{h_2\leq B}\frac{f(h_2)}{h_2^2}.
$$ 
This leads to the upper bound
$N^\circ(B)\ll B(\log B)^{\max\{\lambda,2\}}$.
\end{proof}

\subsection{Proof of the lower bound in Theorem~\ref{theorem:main}}
\begin{lemma}\label{lemma:cubiclb}
In the notation of Theorem~\ref{theorem:main}, we have
$N^\circ(B)\gg B$.
\end{lemma}
\begin{proof}
For each positive integer $x_3$ that is coprime to $\Delta(C)$, 
it follows from~\eqref{eq:hensel} and the Chinese remainder theorem that
there are $\rho(x_3^2)=\rho(x_3)$ solutions $\alpha\bmod x_3^2$ to the congruence 
$$
C(x,1)\equiv 0\bmod x_3^2.
$$
Each $\alpha$ give rise to a lattice generated by $(x_3^2,0)$ and $(\alpha,1)$, which produces solutions to $C(x_1,x_2)\equiv 0\bmod x_3^2$.
In each lattice, Minkowski's theorem implies that
$C(x_1,x_2)\equiv 0\bmod x_3^2$
has at least one non-trivial solution $(x_1,x_2)$ such that $x_1,x_2\ll x_3$.
With such a solution, write
\[C(x_1,x_2)=x_3^2x_4.\]
Then $x_3^2x_4=C(x_1,x_2)\ll x_3^3$, so $x_4\ll x_3$.
Varying $x_3$ over integers between $1$ and $B$ such that $x_3$ is coprime to $\Delta(C)$, it follows from Lemma~\ref{lemma:meanf} that
\[\sum_{\substack{1\leq x_3\leq B\\\gcd(x_3,\Delta(C))=1}}\rho(x_3)\gg B
\]
which thereby gives the lower bound.
\end{proof}

Theorem~\ref{theorem:main} follows immediately from Lemmas~\ref{lemma:cubicub} and~\ref{lemma:cubiclb}.

\section{Integral points on quadratic twists with large $\gcd(x,D)$}
Fix integers $A,B$ such that $x^3+Ax+B$ is irreducible over $\Q$.
Define
\[E_D:y^2=x^3+AD^2x+BD^3,\]
where $D\in \D$.
The discriminant of $E_D$ is
\[\Delta(E_D)=-16(4A^3+27B^2)D^6.\]
We henceforth take $C$ to be the binary cubic form such that
\[C(x,D)=x^3+AD^2x+BD^3.\]
Note that $x^3+Ax+B$ being irreducible implies that $C$ is also irreducible.
\begin{lemma}
Assume that $C(x_1,x_2)\in \Z[x_1,x_2]$ is irreducible.
Recall the definition of $\varrho$ from~\eqref{eq:defrho}.
Then
\begin{equation}\label{eq:posrhomertens}
\sum_{\substack{p\leq N\\ \varrho(p)\geq 1}}\frac{1}{p}\leq \frac{2}{3}\log\log N+O(1)\end{equation}
and
\begin{equation}\label{eq:posrhodensity}
\#\{n\leq N: \varrho(n)\geq 1\}\ll N(\log N)^{-\frac{1}{3}}.\end{equation}
\end{lemma}
\begin{proof}
Since $C$ is irreducible, the Galois group of $C(x,1)$ is either $S_3$ or $C_3$.
The density of primes such that $\varrho(p)\geq 1$ is $\frac{2}{3}$ when the Galois group of is $S_3$, and $\frac{1}{3}$ instead when the Galois group is $C_3$.
The first claim~\eqref{eq:posrhomertens} is a consequence of the Chebotarev density theorem. The second claim~\eqref{eq:posrhodensity} follows from~\eqref{eq:posrhomertens} and \cite[Satz~1]{Wirsing}.
\end{proof}

We first show that for almost all $D\in\D$, there is no integral point $(x,y)\in E_D(\Z)$ with large $\gcd(x,D)$.

\begin{lemma}\label{eq:largegcd}
Let $\kappa>0$.
We have
$$
\#\left\{D\in\D(N): 
\begin{array}{l}
\gcd(x(P),D)\geq N(\log N)^{-\kappa}\\
\text{for some }P\in E_D(\Z)
\end{array}
\right\}
\ll_{\kappa} N(\log N)^{-\frac{1}{3}}\log\log N.
$$
\end{lemma}
\begin{proof}
Suppose $(x,y)\in E_D(\Z)$ is such that $g\coloneqq \gcd(x,D)\geq N(\log N)^{-\kappa}$.
Write $x=g\tilde{x}$, $D=g\tilde{D}$. 
Then from the equation for $E_D$, we see that $g^3\mid y^2$. Since $D$ is square-free, $g$ must also be square-free, so $g^2\mid y$. Write $y=g^2\tilde{y}$. Then we can rewrite the equation as 
$$
g\tilde{y}^2=\tilde{x}^3+A\tilde{D}^2\tilde{x}+B\tilde{D}^3=C(\tilde{x},\tilde{D}).
$$
Since $D$ is square-free, we have that $\gcd(g,\tilde{D})=1$. Thus we deduce 
that $C(\tilde{x}/\tilde{D},1)\equiv 0\bmod g$, whence $\varrho(g)\geq 1$.

Factorise $|D|$ uniquely as a product $|D|=ab$, where $a$ is the product of all prime divisors of $D$ such that $\varrho(p)\geq 1$. Then $g$ must be a divisor of $a$, so $b\leq N/g\leq (\log N)^{\kappa}$.
Using~\eqref{eq:posrhodensity} to bound the number of $a$, we get that the number of possible $D$ is bounded by
\[2\sum_{b\leq (\log N)^{\kappa}}\sum_{\substack{ a\leq N/b\\ \varrho(a)\geq 1}}1
\ll
\sum_{b\leq (\log N)^{\kappa}}\frac{N/b}{(\log (N/b))^{\frac{1}{3}}}
\ll_{\kappa} \frac{N\log\log N}{(\log N)^{\frac{1}{3}}},
\]
as claimed.
\end{proof}

\section{Integral points on quadratic twists with small $\gcd(x,D)$}

\subsection{Preliminaries}
Consider a binary quartic form of the form 
\[f(X,Y)=a_0X^4+4a_1X^3Y+6a_2X^2Y^2+4a_3XY^3+a_4Y^4,\quad a_i\in\Z.\]
The invariants of $f$ are
\begin{align*}
I=I(f)&=a_0a_4-4a_1a_3+3a_2^2,\text{ and }\\
J=J(f)&=a_0a_2a_4-a_0a_3^2-a_1^2a_4+2a_1a_2a_3-a_2^3.
\end{align*}
The discriminant of $f$ is 
\begin{align*}
\Delta (f)=~& I^3-27J^2\\
 =~& a_0^3 a_4^3- 64 a_1^3a_3^3- 18 a_0^2 a_2^2 a_4^2 - 12 a_0^2a_1 a_3 a_4^2 - 6 a_0 a_1^2 a_3^2 a_4\\&\quad
 - 180 a_0 a_1 a_2^2 a_3 a_4 
 + 81 a_0 a_2^4 a_4
 + 36 a_1^2a_2^2 a_3^2
 - 27 (a_0^2 a_3^4+a_1^4a_4^2) \\&\quad
 + 54 a_2 (-a_2^2 + 2 a_1 a_3 + a_0 a_4) (a_4 a_1^2 + a_0 a_3^2).
\end{align*}
The seminvariants attached to the form are $I$, $J$, $a=a(f)=a_0$,
$$H=H(f)=a_1^2-a_0a_2,$$
and 
$$
R=R(f)=2a_1^3+a_0^2a_3-3a_0a_1a_2.
$$
In particular $H$ is the leading coefficient of the quartic covariant
\begin{equation}\label{eq:Gf}
\begin{split}
G_f(X)=~& (a_1^2 - a_0 a_2)X^4 +2 (a_1 a_2 - a_0 a_3)X^3\\
&\quad + (3 a_2^2 - a_0a_4 - 2 a_1 a_3)X^2 +2 (a_2 a_3 - a_1 a_4)X+(a_3^2 - a_2 a_4),
\end{split}
\end{equation}
and $R$ is the leading coefficient of the sextic covariant of $f$. Comparing to the formulas in~\cite[Section~4.1.1]{Cremona}, here we have removed a factor of $-48$ from $H$ and the quartic covariant, a factor of $32$ from $R$, a factor of $12$ from $I$, a factor of $432$ from their $J$, and a factor of $256\cdot 27$ from $\Delta$.

Given $\gamma=
(\begin{smallmatrix} a& b\\c&d\end{smallmatrix})\in\SL_2(\Z)$, define 
$\gamma\cdot (X,Y)=(aX+bY,cX+dY)$
and
\[(\gamma\cdot f)(X,Y)=f(\gamma\cdot (X,Y))=f(aX+bY,cX+dY).\]
Furthermore, define the action of $\gamma\in\SL_2(\Z)$ on $(f,(\alpha,\beta))$ by
\begin{equation}\label{eq:action}
\gamma\cdot (f,(\alpha,\beta))=(\gamma\cdot f,\gamma^{-1}\cdot(\alpha,\beta)),\end{equation}
and observe that the value of $f(\alpha,\beta)$ is preserved under this action.

\subsection{Quartic forms associated to integral points}

Given $P=(x_0,y_0)\in E_D(\Z)$, we may write down a corresponding quartic form
\begin{equation}\label{eq:fp}
f_P(X,Y)=X^4-6x_0X^2Y^2+8y_0XY^3+(-4AD^2-3x_0^2)Y^4.
\end{equation}
Here $\Delta(f)=\Delta(E_D)=-16(4A^3+27B^2)D^6$, $I(f)=-4AD^2$ and $J(f)= -4BD^3$. This construction is due to Mordell \cite[Chapter~25]{Mordell}. (See also \cite[Section~2.2]{AlpogeHo}.)

\begin{lemma}\label{lemma:defFP}
Suppose $P=(x_0,y_0)\in E_D(\Z)$.
Let $M$ be any integer such that $M\mid D$ and $\gcd(M,2x_0)=1$. Take any integer $k$ such that $k\equiv y_0x_0^{-1} \bmod |M|$.
Then
\begin{equation}\label{eq:defFP}
 F_P(X,Y)=\frac{1}{M^3}f_P(MX+kY,Y)\end{equation}
 is integral and satisfies
\begin{itemize}
 \item $F_P(1,0)=M$;
 \item $I=-4A(D/M)^2$ and $J=-4B(D/M)^3$;
 \item $\Delta(F)=\Delta(f)/M^6=-16(4A^3+27B^2)(D/M)^6$.
\end{itemize}
\end{lemma}
\begin{proof}
It suffices to show that there exists an integer $k\equiv y_0x_0^{-1} \bmod |M|$ such that the properties of $F_P$ hold, because any other choice of $k$ would simply transform $F_P$ by $(\begin{smallmatrix}1& *\\0&1\end{smallmatrix})\in\SL_2(\Z)$, which preserves the integrality of $F_P$ and its invariants.

Since $M\mid D$, we have $y_0^2\equiv x_0^3+AD^2x_0\bmod |M|^3$ from the equation for $E_D$. If $k\equiv y_0x_0^{-1}\bmod |M|$, then $x_0\equiv (y_0x_0^{-1})^{2}\equiv k^2\bmod |M|$. By Hensel lifting, there exists some $k$ such that $k\equiv y_0x_0^{-1}\bmod |M|$ and 
\begin{equation}\label{eq:forx0}
x_0\equiv k^2\bmod |M|^3,
\end{equation} 
provided $M$ is odd.
Now since $y_0\equiv kx_0\equiv k^3\bmod |M|$ and $y_0^2\equiv x_0^3\equiv k^6\bmod |M|^2$, we deduce that $y_0\equiv k^3\bmod |M|^2$.
Then, on solving 
$$
(k^3+\lambda M^2)^2\equiv y_0^2\equiv x_0^3+AD^2x_0\equiv k^6+AD^2k^2 \bmod |M|^3,
$$ 
for $\lambda\in \ZZ/M\ZZ$, we see that 
\begin{equation}\label{eq:fory0}
y_0\equiv k^3+\frac{AD^2}{2k}\bmod |M|^3.
\end{equation}
Using~\eqref{eq:forx0} and~\eqref{eq:fory0}, we can check that the new quartic form
\begin{equation}
\begin{split}
 F_P(X,Y)=~&\frac{1}{M^3}f_P(MX+kY,Y)\\
 =~&MX^4+4kX^3Y+\frac{6(k^2-x_0)}{M}X^2Y^2
\\&\quad
+\frac{4(k^3-3x_0k+2y_0)}{M^2}XY^3+\frac{k^4-6x_0k^2+8y_0k-4AD^2-3x_0^2}{M^3}Y^4,\label{eq:FP}
\end{split}\end{equation}
is integral and has all the claimed properties.
\end{proof}

\begin{lemma}\label{lemma:injectivemap}
Fix $A, B\in\Z$ such that $4A^3+27B^2\neq 0$. Fix a choice of integer $M\mid \frac{D}{\gcd(2x(P),D)}$, for every $P\in \bigcup_{D\in\D}E_D(\Z)$.
Then the map from
\[P\in \bigcup_{D\in\D}E_D(\Z)\]
to
\[ (F_P,(1,0))/\SL_2(\Z),\]
with $F_P$ as defined in~\eqref{eq:defFP} and the $\SL_2(\Z)$-action as defined in~\eqref{eq:action}, is well-defined and injective.
\end{lemma}
\begin{proof}
We first prove that the map is well-defined.
Given $P\in E_{D}(\Z)$ with the prescribed choice of $M$, suppose $k$ and $k'$ are two choices of integer such that $k\equiv k'\equiv y_0x_0^{-1}\bmod |M|$, so $k'=k+bM$ for some $b\in\Z$.
Then, with such choices, the corresponding $F_P$ is 
$$
F'(X,Y)=\frac{1}{M^3}f_P(MX+k'Y,Y) 
\quad \text{ and } \quad F(X,Y)=\frac{1}{M^3}f_P(MX+kY,Y).
$$ 
But
\[(F',(1,0))=\begin{pmatrix}1&b\\0&1\end{pmatrix}\cdot (F,(1,0)),\]
so $(F',(1,0))$ and $(F,(1,0))$ are in the same $\SL_2(\Z)$-equivalence class. Therefore we have checked that the image $(F_P,(1,0))$ does not depend on the choice of $k$, so the map is well-defined.

We now prove that the map is injective. 
Suppose $P,Q\in \bigcup_{D\in\D}E_{D}(\Z)$ are such that
\begin{equation}\label{eq:sameclass}
(F_P,(1,0))=\gamma\cdot (F_Q,(1,0)).\end{equation}
The condition that $(1,0)=\gamma^{-1}\cdot (1,0)$ implies that the first column of $\gamma$ must be $(1,0)$, and so 
\[\gamma=\begin{pmatrix}
1&b\\
0&1
\end{pmatrix}\]
for some $b\in\Z$, since $\gamma\in\SL_2(\Z)$ must have determinant $1$.
Putting this back into~\eqref{eq:sameclass}, we have
\begin{equation}\label{eq:classFP}
F_P(X,Y)=(\gamma\cdot F_Q)(X,Y)=F_Q(X+bY,Y).\end{equation}
Observe from the properties given in Lemma~\ref{lemma:defFP} that
$F_P(1,0)=F_Q(1,0)$ determines $M$,
and $J(F_P)=J(F_Q)$ determines $D/M$,
so it must be that $P,Q\in E_{D}(\Z)$ for the same $D\in\D$.
Write
\begin{equation}\label{eq:defFPQ}F_P(X,Y)=\frac{1}{M^3}f_P(MX+k_PY,Y)\quad\text{ and }\quad F_Q(X,Y)=\frac{1}{M^3}f_Q(MX+k_QY,Y),\end{equation}
where $k_P,k_Q\in\Z$ and $M=F_P(1,0)=F_Q(1,0)$.

 Combining~\eqref{eq:defFPQ} and~\eqref{eq:classFP}, we have
\[f_P(MX+k_PY,Y)=f_Q(M(X+bY)+k_QY,Y).\]
On replacing $X$ by $(X-k_PY)/M$, we deduce that
\begin{equation}\label{eq:fPQcompare}
f_P(X,Y)=f_Q(X+(bM+k_Q-k_P)Y,Y).\end{equation}
Recall that the $X^3Y$-coefficients of $f_P$ and $f_Q$ are $0$ by construction. Comparing the $X^3Y$-coefficients on the two sides of~\eqref{eq:fPQcompare}, we have
\[bM+k_Q-k_P=0.\]
Therefore 
$f_P(X,Y)=f_Q(X,Y)$, and
hence $P=Q$.
\end{proof}
\subsection{Bounding the number of equivalence classes of quartic forms}
We apply Lemma~\ref{lemma:injectivemap} with 
\[M=\frac{D}{\gcd(2x(P),D)},\]
for every $P\in\bigcup_{D\in\D} E_D(\Z)$. To bound the size of the image of $\bigcup_{D\in\D(N)} E_D(\Z)$ with $\gcd(x(P),D)\leq N(\log N)^{-\kappa}$, we bound the number of $\SL_2(\Z)$-equivalence classes of binary quartic forms $F$ that can arise, and we will bound separately the number of solutions $(\alpha,\beta)\in\Z^2$ to the Thue inequality $|F(\alpha,\beta)|\leq N/G$, that comes from placing a dyadic interval of length $G$ around 
$\gcd(x(P),D)$.

To bound the number of $\SL_2(\Z)$-equivalence classes of binary quartic forms, we appeal to reduction theory, which allows us to choose a representative in each $\SL_2(\Z)$-equivalence class with bounded seminvariants. 
In doing so, we will make use of the following bounds taken from work of Cremona \cite{Cremona}.
\begin{lemma}[{\cite[Propositions~11 and~14]{Cremona}}]\label{lemma:redquart}
Every integral binary quartic form $F$ with non-zero discriminat is $\SL_2(\Z)$-equivalent to a form $F_{\mathrm{red}}$ with seminvariants in the range 
\[a(F_{\mathrm{red}})\ll |\phi|+|I(F_{\mathrm{red}})|^{\frac{1}{2}}\quad\text{ and }\quad H(F_{\mathrm{red}})\ll |\phi|^2+|I(F_{\mathrm{red}})|,\]
where
$\phi$ denotes the real root of $X^3-\frac{I(F)}{4}X-\frac{J(F)}{4}$ with the largest absolute value.
\end{lemma}

We now reduce the problem about counting $\SL_2(\Z)$-equivalence classes of quartic forms to a problem about counting points on a cubic surface.
\begin{lemma}\label{lemma:countforms}
The number of $\SL_2(\Z)$-equivalence classes of $F_P$ in the image of 
$$
\{P\in E_D(\Z):D\in\D,\ \gcd(x(P),D)< G\}
$$ 
is bounded by
the number of $(h,a,r,g)\in \Z^4$, where $g$ is a positive square-free integer, and 
\begin{equation}\label{eq:rangetwists}
h,a,r, g\ll_{A,B}G,
\end{equation}
such that 
\begin{equation}\label{eq:surfacetwists}
h^3+Aa^2h+B a^3=r^2g.
\end{equation}
Furthermore we can assume that $r\neq 0$ with $O_{A,B}(G(\log G)^6)$-many exceptions.
\end{lemma}

\begin{proof}
For $P=(x_0,y_0)\in E_D(\Z)$, write $g= \gcd(2x_0,D)$, so that $D=Mg$. Reduce every $F=F_P$ to an integral binary quartic form $F_{\mathrm{red}}$ with seminvariants in the bounded range given in Lemma~\ref{lemma:redquart}.
It follows from Lemma~\ref{lemma:defFP} that
$I(F)=-4Ag^2$ and $J(F)=-4Bg^3$.
Take $\phi_1$ to be the real root of $X^3+AX+B$ with the largest absolute value.
Then $\phi=\phi_1 g$ is the real root of 
$$
X^3-\frac{1}{4}I(F)X-\frac{1}{4}J(F)=X^3+Ag^2X+Bg^3
$$ 
with largest absolute value. Since $\phi\ll_{A,B}g$, Lemma~\ref{lemma:redquart} implies that
\begin{equation}\label{eq:reducedquart}
 a(F_{\mathrm{red}})\ll_{A,B}g\quad \text{ and }\quad H(F_{\mathrm{red}}) \ll_{A,B}g^2.
\end{equation}

To count the number of $F/\SL_2(\Z)$, it suffices to count the number of possible tuples $(I,J,a,H,R)$ taken by $F_{\mathrm{red}}$, such that the tuple satisfies the syzygy
\begin{equation}\label{eq:syzygy}
H^3-\frac{I}{4}a^2H-\frac{J}{4}a^3=\left(\frac{R}{2}\right)^2,
\end{equation}
where $(I,J,a,H,R)=(I(F_{\mathrm{red}}),J(F_{\mathrm{red}}),a(F_{\mathrm{red}}),H(F_{\mathrm{red}}),R(F_{\mathrm{red}}))$.
Plugging in $I(F_{\mathrm{red}})=I(F)=-4Ag^2$ and $J(F_{\mathrm{red}})=J(F)=-4Bg^3$, \eqref{eq:syzygy} becomes
\begin{equation}\label{eq:HR}
H^3+Ag^2a^2H+Bg^3a^3=\left(\frac{R}{2}\right)^2.
\end{equation}

We can check, by putting~\eqref{eq:FP} into~\eqref{eq:Gf}, that 
\begin{align*}
G_F(X)=~&x_0X^4 +\frac{4}{M}(- y_0 + x_0 k)X^3
+ \frac{2}{M^2}(2A D^2 + 3 x_0^2 + 3 x_0 k^2 - 6 y_0 k)X^2 \\
&+\frac{4}{M^3} ( A D^2 k + 3 x_0^2 k - x_0y_0 + x_0k^3 - 3 y_0 k^2)X\\
&+\frac{1}{M^4}(-4 A x_0 D^2 + 4 A D^2 k^2 - 3 x_0^3 + 6 x_0^2 k^2 - 4x_0y_0 k + x_0 k^4 + 4 y_0^2 - 4 y_0 k^3).
\end{align*}
Every coefficient of $G_F(X)$ is divisible by $g$, so $g\mid G_F(X)$. Then we also know that $g\mid G_{F_{\mathrm{red}}}(X)$ because $G_F$ is a covariant of $F$.
Hence $H$, being the leading coefficient of $ G_{F_{\mathrm{red}}}(X)$, must also be divisible by $g$.
The left hand side of~\eqref{eq:HR} is integral, so $R$ must be even.
From~\eqref{eq:HR}, we deduce that $g^3\mid \left(\frac{R}{2}\right)^2$, so $g^2\mid \frac{R}{2}$ since $g$ is square-free. Therefore, on writing $H=gh$ and $\frac{R}{2}=g^2r$, we see that~\eqref{eq:HR} becomes~\eqref{eq:surfacetwists}.
 The bounds~\eqref{eq:reducedquart} on the variables becomes
\[a\ll_{A,B}g \quad\text{ and }\quad h\ll_{A,B}g.\]
When $g<2G$, this implies that $(h,a,r,g)$ satisfies~\eqref{eq:rangetwists}.
Since $(h,a,r,g)$ determines $(I,J,a,H,R)$, it suffices to bound the number of such $(h,a,r,g)$.

If $r=R=0$, then $(H,\frac{1}{2}R)$ is a torsion point in $E_{ga}(\Z)$, so $ga\mid H$. Writing $F_{\red}(X,Y)=a_0X^3+a_1X^2Y+a_2XY^3+a_3Y^3$, we have $ga_0\mid H=a_1^2-a_0a_2$, hence $a_0\mid a_1^2$.
Since $a_0$ or $a_1^2$ divides every term in the formula of $\Delta(F_{\red})$, we see that $a_0\mid \Delta(F_{\red})=-16(4A^3+27B^2)g^6$.
Given each $g$, the number of possible $a=a_0$ is $\ll_{A,B} 7^{\omega(g)}$. There are at most 3 torsion point in $E_{ga}(\Z)$, so the contribution of such forms is $\ll_{A,B} \sum_{g\ll G}7^{\omega(g)}\ll G(\log G)^6$.
\end{proof}

We record the following upper bound, due to Thunder \cite{Thunderdecomp}, which allows us to count the number of solutions to Thue inequalities.
\begin{lemma}\label{lemma:thueineq}
Suppose that $F(X,Y)$ is an integral binary form with non-zero discriminant. Then for any positive integer $m$, we have
\[\#\{(X,Y)\in\Z^2:|F(X,Y)|\leq m\}\ll m^{\frac{2}{\deg(F)}},\]
where the implied constant only depends on the degree $\deg(F)$ of $F$.
\end{lemma}
\begin{proof}
To apply \cite[Theorem~2]{Thunderdecomp}, it suffices to check that the area 
\[A_F=\vol\{(X,Y)\in\R^2:|F(X,Y)|\leq 1\}\] is finite.
The finiteness of $A_F$ follows from work of Bean \cite[Corollary~1]{Bean}.
\end{proof}

We are now ready to bound the contribution from integral points $(x,y)\in E_D(\Z)$ with small $\gcd(x,D)$.
\begin{lemma}\label{lemma:smallgcd}
For any $K\geq 1$, we have 
$$
\sum_{D\in\D(N)} \#\{P\in E_D(\Z):\gcd(x(P),D)<K\}\ll (NK)^{\frac{1}{2}}(\log K)^{\nu} ,
$$
where $\nu$ is $3$ when $E_D[2]$ is trivial over $\Q$ and $6$ otherwise.
\end{lemma}
\begin{proof}
We split the sum into dyadic intervals according to the size of $\gcd(x(P),D)$. Consider the points $P\in \bigcup_{D\in\D(N)} E_D(\Z)$ such that $G\leq \gcd(x(P),D)<2G$.

It follows from Lemma~\ref{lemma:injectivemap} that we want to bound the number of $\SL_2(\Z)$-equivalence classes of $(F_P,(1,0))$ from such $P$. Given any integral binary quartic form $F$, we can collect all $F_P$ that are $\SL_2(\Z)$-equivalent to $F$, which we write as $F_P\sim_{\SL_2(\ZZ)} F$.
 This allows us to transform $(F_P,(1,0))$ into $(F,(\alpha,\beta))$, where $(\alpha,\beta)\in \Z^2$.
Then, since 
$$
F(\alpha,\beta)=F_P(1,0)=M=\frac{D}{\gcd(2x(P),D)}, 
$$
we see that $(\alpha,\beta)$ satisfies $|F(\alpha,\beta)|\leq N/G$. By Lemma~\ref{lemma:thueineq}, the number of solutions to the inequality
$|F(X,Y)|\leq N/G$ is $\ll (N/G)^{1/2}$. Therefore, given any $F$, we have the upper bound
\begin{equation}\label{eq:pointsinclass}
\sum_{D\in\D(N)}\#\{P\in E_D(\Z):F_P\sim_{\SL_2(\Z)}F,\ \gcd(x(P),D)\geq G\}\ll \left(\frac{N}{G}\right)^{1/2}.
\end{equation}

Next we bound the number of $\SL_2(\Z)$-equivalence classes of $F$ that can arise as the image of some $P\in \bigcup_{D\in\D} E_D(\Z)$. By Lemma~\ref{lemma:countforms}, it suffices to count $(h,a,r,g)\in\Z^4$ such that $g$ is square-free, $h,a,g,r\ll G$ and $h^3+Aa^2h+B a^3=gr^2$.
First suppose that $r\neq 0$.
Since $g$ is square-free, it must be that $b\coloneqq \gcd(h,a,g)\mid r$.
Hence we can bound the number of $(h/b,a/b,r/b,g/b)$ using Theorem~\ref{theorem:main}, because $\gcd(h/b,a/b,g/b)=1$. Summing over $b\leq G$, we thereby obtain
 \begin{align*}
 \#\left\{(h,a,r,g)\in\Z^4: 
 \begin{array}{l}
 g \text{ square-free},\ h,a,g,r\ll G\\ 
 h^3+Aa^2h+B a^3=gr^2,\ r\neq 0
 \end{array}
 \right\}
 &\ll \sum_{b\leq G} \frac{G}{b}(\log G)^{\max\{\lambda,2\}}\\
 &\ll G(\log G)^{\max\{\lambda,2\}+1}.
 \end{align*}
 If $E_D$ has non-trivial two torsion, then we also have to include $O(G(\log G)^6)$-many exceptions to account for the possibility that $r=0$.
Therefore the number of $\SL_2(\Z)$-equivalence classes in
\[\bigcup_{D\in\D} \{F_P:P\in E_D,\ \gcd(x(P),D)<2G\}\]
is bounded by $\ll G(\log G)^{\nu}$.

Summing~\eqref{eq:pointsinclass} over the $O(G(\log G)^{\nu})$-many $\SL_2(\Z)$-equivalence classes of $F$, we have
\begin{equation}\label{eq:Gbound}
\begin{split}
 \sum_{D\in\D(N)} \#\{P\in E_D(\Z):G\leq \gcd(x(P),D)<2G\}
 &\ll G(\log G)^{\nu}\left(\frac{N}{G}\right)^{\frac{1}{2}}\\ &=
(NG)^{\frac{1}{2}}(\log G)^{\nu} .
\end{split}
\end{equation}
Finally, summing over $G$ being powers of $2$ and subject to $1\leq G<K$, gives the desired result.

\end{proof}

\begin{remark}
Heuristically, we expect the set $\{(X,Y)\in\Z^2:|F(X,Y)|\leq m\}$ considered in Lemma~\ref{lemma:thueineq} to have size $\sim A_Fm^{\frac{2}{\deg(F)}}$, where 
$$
A_F=\vol \{(X,Y)\in\R^2:|F(X,Y)|\leq 1\}\ll \Delta(F)^{-\frac{1}{\deg(F)(\deg(F)-1)}},
$$ 
by \cite[Corollary~1]{Bean}.
When $F$ is of degree $4$ with discriminant $\Delta(F)=-16(4A^3+27B^2)g^6$, this gives an upper bound of $\ll \Delta(F)^{-\frac{1}{12}}m^{\frac{1}{2}}\ll_{A,B} g^{-\frac{1}{2}}m^{\frac{1}{2}}$.
Such a bound would replace~\eqref{eq:pointsinclass} by $(N/G^2)^{\frac{1}{2}}$, and replace~\eqref{eq:Gbound} by $\ll N^{\frac{1}{2}}(\log G)^{\nu}$. Summing over dyadic intervals of length $G\leq N$ would then give 
$$
\sum_{D\in\D(N)} \#E_D(\Z)\ll N^{\frac{1}{2}}(\log N)^{\nu+1}.
$$ 
However, in general, we cannot expect to be able to prove strong point-wise bounds of the form $\ll A_Fm^{\frac{2}{\deg(F)}}$ on the number of solutions to Thue inequalities when $m$ is small relative to $\Delta(F)$. As we will see in Lemma~\ref{lemma:twistslb}, there are infinitely many integral points in the family $E_D$, which implies that there are integral binary quartic forms $f_P$ taking the form~\eqref{eq:fp} with arbitrary large $\Delta(f_P)$, and $f_P(1,0)=1$ .
\end{remark}

\subsection{Completing the proof of Theorem~\ref{theorem:maintwists}}\label{s:5.4}

Taking $K=N(\log N)^{-\kappa}$ in Lemma~\ref{lemma:smallgcd}, we obtain
\[\sum_{D\in\D(N)}\# \{P\in E_D(\Z):\gcd(x(P),D)<N(\log N)^{-\kappa}\}\ll N(\log N)^{3-\frac{1}{2}\kappa}.
\]
The upper bound in Theorem~\ref{theorem:maintwists} therefore follows from
combining this with Lemma~\ref{eq:largegcd}, on taking $\kappa=7$.
The lower bound is achieved in the following result. 

\begin{lemma}\label{lemma:twistslb}
Fix $A,B\in\Z$ such that $x^3+Ax+B$ is separable over $\Q$.
Then
\[\#\{D\in\D(N): E_D(\Z)\setminus E_D[2]\neq\varnothing\}\gg_{A,B} N^{\frac{1}{2}}.\]
\end{lemma}
\begin{proof}
As before, take $C$ to be the binary cubic form such that $C(x,1)=x^3+Ax+B$.
Take $F(X,Y)=Y\cdot C(X,Y)$, which 
has  leading coefficient $1$ and 
must have non-zero discriminant, because $C$ is separable over $\Q$.
Appealing to work of Xiao  \cite[Theorem~1.2]{Xiaopfree}, we conclude  that 
\[\#\left\{D\in\D(N): F(\alpha,\beta)=D\text{ for some }(\alpha,\beta)\in \Z^2\right\}\gg_F N^{\frac{1}{2}}.\]
If 
$F(\alpha,\beta)=\beta\cdot C(\alpha,\beta)=D$,
then $\beta\mid D$ and  we may write $D=\beta d$, for a suitable $d\in \NN$.
Hence
$C(\alpha,\beta)=d$, and so 
$C(\alpha d,\beta d)=d^4$, which thereby  gives a point $(\alpha d,d^2)\in E_{D}(\Z)$.
\end{proof}

\section{Character sum input}
The techniques developed by Heath-Brown in \cite{HBSelmer1,HBSelmer} allow us to prove the following result to handle sums of products of quadratic characters. 
Our proof will also incorporate some refinements made in \cite{FK4rank}.
We have chosen to extract a general form of the result since we will need to apply it several times in 
different contexts.

\begin{theorem}\label{theorem:charsum}
Let $\cI$ be a non-empty finite set and fix some subsets $\cJ_1,\dots, \cJ_r $ of $\cI$. Define a function $\Phi:\cI\times\cI\rightarrow \F_2$ such that $\Phi(i,i)=0$ for all $i\in\cI$. Fix $\kappa,\epsilon,C>0$.
For each $i\in\cI$:
\begin{itemize}
\item let $c_i\mid\beta_i$ be positive integers such that  $c_i\leq C$, $\beta_i\leq (\log N)^{\kappa}$;
\item let $K_i/\Q$ be a Galois extension and let $\alpha_i\in K^{\times}$ such that $K_i(\sqrt{\alpha_i})$ is Galois over $\Q$ of degree less than $C$ and $\Disc(K_i(\sqrt{\alpha_i})/\Q)\mid \beta_i$;
\item let $\chi_i$ be a multiplicative function such that for any prime $p$,
\[\chi_i(p)=\begin{cases}
\Big(\dfrac{\alpha_i}{\p}\Big)&\text{if }p
\text{ splits completely in } K_i,\\
\hfil 0&\text{otherwise},
\end{cases}\]
 where $\p$ denotes a prime in $K_i$ lying over $p$;
\item let $f_i$  be a multiplicative function such that for any $q\in\Z/c_i\Z$, there exists $f_{i,q}\in[0,1]$ such that $f_i(p)=f_{i,q}$ for every prime $p\equiv q\bmod c_i$.
\end{itemize} 
Define
\[\lambda\coloneqq \max_{i\in\cI}\frac{1}{\#(\Z/c_i\Z)^{\times}}\sum_{q\in(\Z/c_i\Z)^{\times}}f_{i,q}.\]
Let
 $M$ be the maximum possible size of a subset $\cU\subseteq \cI$ such that $\Phi(i,j)+\Phi(j,i)=0$ for every $i,j\in\cU$.
 Let
 \[S(N)\coloneqq \sum_{(D_i)}
 \prod_{i\in\cI} f_i(D_i)\chi_i(D_i)
 \prod_{i,j\in\cI}\leg{D_i}{D_j}^{\Phi(i,j)},\]
 where the sum is over all tuples of positive integers $(D_i)_{i\in\cI}$ such that $\prod_i D_i\in\D(N)$, $\gcd(\prod_i D_i,2\prod_{i}\beta_i)=1$, and $\prod_{i\in\cI\setminus\cJ_k}D_i\neq 1$ for every $k=1,\dots,r$.
Then
\[S(N)=\sum_{\cU}\sum_{(D_i)_{i\in\cU}}
 \prod_{i\in\cU} f_i(D_i)\chi_i(D_i)
 \prod_{i,j\in\cU}\leg{D_i}{D_j}^{\Phi(i,j)}+O\left(N(\log N)^{(M-1)\lambda-1+\epsilon}\right),\]
 where the sum over $\cU$ is taken over all $\cU\subseteq \cI$ of size $M$ satisfying all of the following:
 \begin{enumerate}[label=(P\arabic*)]
 \item $\cU\not\subseteq \cJ_k$ for every $k=1,\dots,r$.\label{prop:one}
 \item $\Phi(i,j)+\Phi(j,i)=0$ for every $i,j\in\cU$.\label{prop:two}
 \item For any $i\in\cU$ such that $\Phi(i,j)=0$ for all $j\in\cU$, we have $\sqrt{\alpha_i}\in K_i(\zeta_{c_i})$, where $\zeta_{c_i}$ denotes a primitive $c_i$-th root of unity.\label{prop:single}
 \item For any distinct $i,j\in \cU$ such that
 $\Phi(i,k)=\Phi(j,k)$ for all $k\in\cU$,  we have  $\sqrt{\alpha_i\alpha_j}\in K_i(\zeta_{c_i})\cdot  K_j(\zeta_{c_j})$. \label{prop:double}
\end{enumerate}
The implied constant depends at most on $\#\cI,\kappa,\epsilon, C,\lambda$.
\end{theorem}
The assumptions on $f_i$ in Theorem~\ref{theorem:charsum} imply that
 \begin{equation}\label{eq:mertens}
 \sum_{p\leq N}\frac{f_i(p)}{p}\leq \lambda\log\log N+O(1)
 \end{equation}
 for every $i$, by Mertens' theorem.

\begin{definition}[Linked indices and admissible sets]
Let $\cI, \Phi$ be as in the setting of Theorem~\ref{theorem:charsum}.
We say that two indices $i,j\in\cI$ are {\em linked} if $\Phi(i,j)+\Phi(j,i)=1$, and {\em unlinked} if $\Phi(i,j)+\Phi(j,i)=0$.
We say that $\cU\subseteq\cI$ is an {\em unlinked set} of indices if any $i,j\in\cU$ are unlinked.
We say that an unlinked set $\cU$ is {\em admissible} if it satisfies both~\ref{prop:single} and~\ref{prop:double}.\end{definition}

\subsection{Preliminaries}
We recall several results that we will need to prove Theorem~\ref{theorem:charsum}.
The first is a version of the Siegel--Walfisz theorem worked out by 
Goldstein
\cite{Goldstein}.

\begin{lemma}\label{lemma:SW}
Let $\epsilon>0$. Let $K/\Q$ be a Galois extension of degree $n$. Let $\chi$ be a non-trivial primitive finite Hecke character of $K$ with conductor $\mathfrak{f}$ such that $$|\Disc(K/\Q)\cdot \Norm_{K/\Q}(\mathfrak{f})|\leq (\log N)^{\kappa}.
$$ 
Then 
\[
\sum_{\substack{\p\subset\OO_K\text{ prime}\\ \Norm_{K/\Q}(\p)\leq N\\\p\nmid \mathfrak{f}}}
\chi(\p)\ll 
N
\exp\left(-(\log N)^{\frac{1}{3}}\right),
\]
where the implied constant depends only on $n$ and $\kappa$.
\end{lemma}

The next result is on the double oscillation of characters. An important result of this type by Heath-Brown \cite[Theorem~1]{HBcharsums} is enough for most cases, however the $(MN)^{\epsilon}$ term that appears causes problem when $N$ and $M$ are of very different sizes. An alternative form of this result \cite[Lemma~4]{HBSelmer1} removes this issue and is sufficient for our application. The optimal power saving that follows from the argument of \cite[Lemma~4]{HBSelmer1} can be found in \cite[Lemma~15]{FK4rank}, which takes the following form.
\begin{lemma}\label{lemma:largesieve}
 Let $a_m$ and $b_m$ be complex numbers such that  $|a_m|,|b_m|\leq 1$. Then for every $M,N\geq 1$ and for every $\epsilon>0$, we have
 \[\sum_{m\leq M}\sum_{n\leq N} a_mb_n\mu^2(2m)\leg{n}{m}\ll_{\epsilon}MN(M^{-\frac{1}{2}+\epsilon}+N^{-\frac{1}{2}+\epsilon}).\]
\end{lemma}

\begin{lemma}[{\cite[Theorem~1]{Shiu}}]\label{lemma:Shiu}
 Let $f$ be a non-negative multiplicative function that satisfies $f(n)\leq C^{\omega(n)}$, for some constant $C\geq 1$.
 Then we have 
 \[\sum_{X-Y<n\leq X} f(n)\ll_C \frac{Y}{\log X}\exp\left(\sum_{p\leq X}\frac{f(p)}{p}\right)\]
 uniformly for $2\leq X\exp(-\sqrt{\log X})\leq Y<X$.
\end{lemma}

\begin{lemma}[{\cite[Lemma A, p.~265]{HR}}]\label{lemma:HR}
Uniformly for positive integer $\ell$ and $X\geq 2$, we have
\[\#\{n\leq X:\omega(n)=\ell\}\ll \frac{X}{\log X}\cdot \frac{(\log\log X + C_0)^{\ell-1}}{(\ell-1) !},\]
where $C_0$ is an absolute constant.
\end{lemma}

\subsection{Dissection}

Define \[\Omega\coloneqq \lceil e\cdot \#\cI\cdot (\log\log N+C_0)\rceil,\]
where $C_0$ is the constant in Lemma~\ref{lemma:HR}. We will dissect the sum according to the size of each $D_i$ with the dissection parameter 
\[\Delta\coloneqq 1+(\log N)^{-\#\cI}.\] 
For each $i\in\cI$, define a number $A_i$ of the form $1,\Delta,\Delta^2,\dots$. 
For each $\mathbf{A}=(A_i)_{i\in\cI}$,
define the restricted sum
\[
S(N,\mathbf{A})=\sum_{(D_i)}
 \prod_i f_i(D_i)\chi_i(D_i)
 \prod_{i,j}\leg{D_i}{D_j}^{\Phi(i,j)},
\]
where the sum over $(D_i)$ is subject to the conditions \[\gcd(D_i,2\prod_i\beta_i)=1,\quad \omega(D_i)\leq\Omega,\quad
A_i\leq D_i<\Delta A_i,\quad\prod_i D_i\in\D(N),\text{ and }\]
\begin{equation}\label{eq:AsumJ}
\prod_{i\in\cI\setminus\cJ_k}D_i\neq 1\text{ for every }k=1,\dots,r.\end{equation}

If $A_i=1$, the condition $A_i\leq D_i<\Delta A_i$ implies that $D_i=1$.
If $\prod_{i\in\cI\setminus\cJ_k}A_i\neq 1$ for every $k=1,\dots,r$, the condition \eqref{eq:AsumJ} can be dropped from the sum $S(N,\mathbf{A})$. If $\prod_{i\in\cI\setminus\cJ_k}A_i= 1$ holds for some $k$, the condition \eqref{eq:AsumJ} forces $S(N,\mathbf{A})=0$.

The number of possible $A_i$ up to $N$ is bounded by
\[\frac{\log N}{\log \Delta}=\frac{\log N}{\log (1+(\log N)^{-\#\cI})}\ll (\log N)^{1+\#\cI}.\]
Since there are $\#\cI$-many variables, the number of $\mathbf{A}$ such that $S(N,\mathbf{A})$ is non-trivial is 
\[\ll (\log N)^{\#\cI(1+\#\cI)},\]
We also define here two parameters
\[
 N^{\dagger}\coloneqq(\log N)^{4\left(1+\#\cI\cdot (1+\#\cI)\right)}
 \quad 
 \text{ and }
\quad 
N^{\ddagger}\coloneqq\exp\left((\log N)^{\frac{1}{\#\cI}\cdot \epsilon}\right).
\]

\subsection{Number of prime factors of the variables}
We bound the contribution from those $(D_i)$ such that
$\omega(D_i)\geq \Omega$ for some $i\in\cI$.

\begin{lemma}\label{lemma:omegabound}
\[S(N)=
\sum_{\mathbf{A}} S(N,\mathbf{A})+O(N(\log N)^{-1}),
 \]
 where the sum is over $\mathbf{A}$ is such that \[\prod_{i\in\cI} A_i \leq 
N.
 \] 
\end{lemma}
\begin{proof}
We bound each summand of $S(N)$ that does not satisfy $\omega(D_i)<\Omega$ trivially by $1$.
If $\omega(D_i)\geq \Omega$ for some $i\in\cI$, then setting $n=\prod_{i\in\cI}D_i$, we must have $\omega(n)\geq \Omega$. 
Given $n$, each prime factor of $n$ divides one of the $\#\cI$-many $D_i$.
Therefore the contribution from such $n$ is bounded by
\[\ll \sum_{\substack{n\leq N\\ \omega(n)\geq \Omega}}(\#\cI)^{\omega(n)}.\]
By Lemma~\ref{lemma:HR}, we can bound this sum by 
\begin{align*}
&\ll
\frac{N}{\log N}
\sum_{\ell\geq \Omega}
\frac{(\#\cI)^{\ell}(\log\log N+C_0)^{\ell}}{\ell!}\\
&\leq
\frac{N}{\log N}
\frac{(\#\cI)^{\Omega}(\log\log N+C_0)^{\Omega}}{\Omega!}
\sum_{\ell\geq 0}
\frac{(\#\cI)^{\ell}(\log\log N+C_0)^{\ell}}{\Omega^{\ell}}.
\end{align*}
The last sum is a geometric series. Substituting 
$\Omega\coloneqq \lceil e \cdot \#\cI\cdot (\log\log N+C_0)\rceil$ and using Stirling's approximation, we get the upper bound $N(\log N)^{-1}$, as desired.
\end{proof}

\subsection{Incomplete boxes}
We bound the contribution from 
\[\cF_1\coloneqq\left\{\mathbf{A}:\prod_i A_i\geq \Delta^{-\#\cI}N\right\}.\]
Bounding the summands of $S(N,\mathbf{A})$ trivially by $1$, we have
\[\sum_{\mathbf{A}\in\cF_1} |S(N,\mathbf{A})|\leq
 \sum_{\Delta^{-\#\cI}N\leq n\leq N}
(\#\cI)^{\omega(n)}
\ll \left(1-\Delta^{-\#\cI}\right)N(\log N)^{\#\cI-1},
\]
where the last expression follows from Lemma~\ref{lemma:Shiu}.
Since $\Delta=1+(\log N)^{-\#\cI}$,
we have
\[\Delta^{-\#\cI}=\left(1+(\log N)^{-\#\cI}\right)^{-\#\cI}=1-\#\cI(\log N)^{-\#\cI}+O((\log N)^{-2\#\cI}),\]
so inserting this gives
\begin{equation}\label{eq:incomplete}
 \sum_{\mathbf{A}\in\cF_1} |S(N,\mathbf{A})|\ll N(\log N)^{-1}.
\end{equation}

\subsection{Few large indices}

We want to bound the contribution from $\mathbf{A}$ with very few large indices. We will bound the contribution from the set
\[
 \cF_2\coloneqq 
 \left\{\mathbf{A}:\#\{i\in\cI:A_i\geq N^{\ddagger}\}<M\right\}.
\]
Let $m$ be the product of those $D_i\geq N^{\ddagger}$ and let $\cW =\{i\in\mathcal{I}: A_i\geq N^{\ddagger}\}$, then
\[
\sum_{\substack{\mathbf{A}\in\cF_2\\ \#\cW =r}} |S(N,\mathbf{A})|
\leq
\sum_{m\leq (N^{\ddagger})^{\#\cI-r}}
(\#\cI-r)^{\omega(m)}
\sum_{\substack{(D_i)_{i\in\cW }\\ \prod_{i\in\cW } D_i\leq \frac{N}{m}}}
\prod_{\substack{i\in\cW }}f_i(D_i).
\]
Applying Lemma~\ref{lemma:Shiu}, the inner sum becomes
\[\sum_{\substack{(D_i)_{i\in\cW }\\ \prod_{i\in\cW } D_i\leq \frac{N}{m}}}
\prod_{\substack{i\in\cW }}f_i(D_i)\ll \frac{N}{m\log N}\exp\left(\sum_{p\leq N}\sum_{i\in\cW}\frac{f_i(p)}{p}\right)
\ll \frac{N}{m}(\log N)^{r\lambda-1},
\]
where we have applied~\eqref{eq:mertens}.
Then putting this back
\[
\sum_{\substack{\mathbf{A}\in\cF_2\\ \#\cW =r}} |S(N,\mathbf{A})|
\ll
N
(\log N)^{r\lambda-1}
\sum_{m\leq (N^{\ddagger})^{\#\cI-r}}
\frac{(\#\cI-r)^{\omega(m)}}{m}
\ll 
N(\log N)^{r\lambda-1+\epsilon}
. 
\]
Summing over $r\leq M-1$, we have
\begin{equation}\label{eq:fewlarge}
\sum_{\mathbf{A}\in\cF_2} |S(N,\mathbf{A})|\ll
N(\log N)^{(M-1)\lambda-1+\epsilon}.
\end{equation}

\subsection{Two large linked indices}

Let
\[\cF_3\coloneqq\{\mathbf{A}: A_i,\ A_j\geq N^{\dagger}\text{ for some linked } i,j\}\setminus \cF_1.\]
If $i$ and $j$ are linked, exactly one of $\leg{D_i}{D_j}$ and $\leg{D_j}{D_i}$ can appear in the sum.

For any $\mathbf{A}\in\cF_3$, we have
\[|S(N,\mathbf{A})|\\
\ll
\sum_{(D_u)_{u\in\cI\setminus\{i,j\}}}
\left|\sum_{A_i\leq D_i<\Delta A_i}\sum_{A_j\leq D_j<\Delta A_j}
a(D_i)b(D_j)
\leg{D_i}{D_j}\right|,
\]
where
\begin{align*}
a(D_i)&=f_i(D_i)\chi_i(D_i)\prod_{k\neq i,j}\leg{D_i}{D_k}^{\Phi(i,k)}\leg{D_k}{D_i}^{\Phi(k,i)},\\
b(D_j)&=f_j(D_j)\chi_j(D_j)\prod_{k\neq i,j}\leg{D_j}{D_k}^{\Phi(j,k)}\leg{D_k}{D_j}^{\Phi(k,j)}
.
\end{align*}
It follows from the assumptions that $|a(D_i)|,|b(D_j)|\leq 1$.
Apply Lemma~\ref{lemma:largesieve}, then since
$A_i,A_j\geq N^{\dagger}$,
we obtain
\[|S(N,\mathbf{A})|\ll N(N^{\dagger})^{-\frac{1}{2}+\epsilon}.\]
Summing over $O((\log N)^{\#\cI\cdot (1+\#\cI)})$ possible $\mathbf{A}$, we conclude that
\begin{equation}\label{eq:largelinked}
 \sum_{\mathbf{A}\in\cF_3}|S(N,\mathbf{A})|\ll N(N^{\dagger})^{-\frac{1}{2}+\epsilon}(\log N)^{\#\cI\cdot (1+\#\cI)}
\ll N(\log N)^{-1}.
\end{equation}

\subsection{One large and one small linked indices}
Define
\[\cF_4=\{\mathbf{A}: 2\leq A_j<N^{\dagger},\ A_i\geq N^{\ddagger}\text{ for some linked }i,j\}\setminus (\cF_1\cup\cF_3).\]
Any $j$ that is linked to $i$ must satisfy $A_j<N^{\dagger}$ since we assumed that $\mathbf{A}\notin\cF_3$.

Set
\[\chi (D_i)=\chi_i(D_i)\prod_{j\neq i}\leg{D_i}{D_j}^{\Phi(i,j)}\leg{D_j}{D_i}^{\Phi(j,i)}.\]
If $i$ and $j$ are linked, then exactly one of $\leg{D_i}{D_j}$ and $\leg{D_j}{D_i}$ appears non-trivially in the expression.
If $i$ and $j$ are unlinked, then either $\Phi(i,j)=\Phi(j,i)=0$, so neither symbol appears, or $\Phi(i,j)=\Phi(j,i)=1$, in which case we collect the two symbols as
\begin{equation}\label{eq:quad-rec}
\leg{D_i}{D_j}\leg{D_j}{D_i}=(-1)^{\frac{D_i-1}{2}\cdot \frac{D_j-1}{2}}=
\leg{(-1)^{\frac{D_j-1}{2}}}{D_i}.
\end{equation}
Therefore $\chi$ lifts to a character in $K_i$ with modulus
\[4\beta_i\prod_{j\text{ linked to }i}D_j<4\beta_i(N^{\dagger})^{\#\cI-1}<4(\log N)^\kappa(N^{\dagger})^{\#\cI}
,\]
where $\kappa$ is as in the assumptions of Theorem~\ref{theorem:charsum}.
Then
\[S(N,\mathbf{A})\ll
\sum_{(D_j)_{j\in\cI\setminus\{i\}}}
\left|\sum_{D_i}
f_i(D_i)
\chi(D_i)\right|.
\]
Now apply the assumption on the number of prime factors of $D_i$, so
\[
S(N,\mathbf{A})\ll 
\sum_{(D_j)_{j\in\cI\setminus\{i\}}}
\sum_{\ell=1}^{\Omega}
\left|
\sum_{\substack{D_i \\ \omega(D_i)=\ell}}
f_i(D_i)\chi(D_i )
\right|.
\]
Recall from the assumptions that $f_i(p)$ only depends on the value of $p\bmod c_i$.
Write $D_i=p_1\cdots p_{\ell}$ and $p_1<p_2<\dots<p_\ell$. 
Then $\ell\leq \Omega$ and 
we have
$$
S(N,\mathbf{A})\ll 
\sum_{(D_j)_{j\in\cI\setminus\{i\}}}
\sum_{\ell=1}^{\Omega}
\sum_{p_1,\dots, p_{\ell-1}}
\max_{b\in(\Z/4c_i\Z)^{\times}}\left|
\sum_{\substack{p_{\ell}\equiv b\bmod 4c_i\\ p_\ell\nmid \prod_j \beta_j}} \mu^2\left(2p_1\dots p_{\ell}\prod_{j\in\cI\setminus\{i\}}D_j\right) \chi(p_{\ell})
\right|,
$$
where
\[(N^{\ddagger})^{\frac{1}{\Omega}}<( A_i)^{\frac{1}{\ell}}
<p_{\ell}
<\frac{\Delta A_i}{p_1\cdots p_{\ell-1}}
.\]

We use a sum of Dirichlet characters $\psi$ mod $4c_i$ to detect the condition $p_{\ell}\equiv b\bmod 4c_i$. 
Take $\chi'$ to be the character in $K_i$ such that $\chi'(\p)=\chi\psi(\Norm_{K_i/\Q}(\p))$ for any degree $1$ prime $\p$ in $K_i$. At any prime $p$ that splits completely in $K_i/\Q$, the value of $\chi'$ is independent of the choice of $\p$ above $p$.
The number of prime ideals $\p$ in $K_i$ of degree greater than $1$ with norm up to $X$ is bounded by $O(X^{\frac{1}{2}})$, since then $p^2\mid \Norm_{K_i/\Q}(\p)$. Recall that $\chi$ is $0$ at any prime $p$ that does not split completely in $K_i$. Hence we obtain
\[\sum_{p\leq X} \chi(p)\psi(p)=\frac{1}{[K_i:\Q]}\sum_{\substack{\p\subset\OO_K\text{ prime}\\ \Norm_{K_i/\Q}(\p)\leq X}}
\chi'(\p)
+O(X^{\frac{1}{2}}).
\]
Notice that $\chi'$ must be non-trivial because $D_j\geq 2$ for some $j$ linked to $i$, and by assumption $D_j$ is coprime to $4c_i$.

Now apply Lemma~\ref{lemma:SW}, noting that
the modulus of the character $\chi'$ is less than $4(\log N)^{\kappa}(N^{\dagger})^{\#\cI}=
4(\log N)^{\kappa+4\#\cI(1+\#\cI\cdot(1+\#\cI))}$.
  Then, for any $B>0$, we obtain
\[\sum_{p_{\ell}\equiv \kappa\bmod 4c_i}\mu^2\left(2p_1\dots p_{\ell}\prod_{j\in\cI\setminus\{i\}}D_j\right)\chi(p_{\ell}) 
\ll _{B}
\frac{A_i}{p_1\cdots p_{\ell-1}}\left(\frac{1}{\Omega}\log N^{\ddagger}\right)^{-B}+ \Omega,
\]
where the last term  comes from those $p_\ell$ that  divide some
$2p_1\cdots p_{\ell-1}\prod_{j\neq i} D_j\prod_j \beta_j$.
Summing over $p_{\ell-1}, \ldots, p_1$, and then over $(D_j)_{j\in\cI\setminus\{i\}}$, we have
\[
S(N,\mathbf{A})
\ll_{B} N (\log N)\left(\frac{1}{\Omega}\log N^{\ddagger}\right)^{-B}.
\]
Picking $B$ large enough and  summing over  $O((\log N)^{\#\cI\cdot (1+\#\cI)})$ possible $\mathbf{A}$, we finally deduce that 
\begin{equation}\label{eq:largesmalllinked}
 \sum_{\mathbf{A}\in\cF_4}|S(N,\mathbf{A})|
\ll N(\log N)^{-1}.
\end{equation}
\subsection{Inadmissible sets}
Define
\[\cF_5\coloneqq \bigcup_{\substack{\cU\text{ unlinked }\\\text{inadmissible}}}\cF_5(\cU),\]
where
\[\cF_5(\cU)\coloneqq\{\mathbf{A}: A_k=1\text{ for all }k\notin\cU,\ A_i\geq N^{\ddagger}\text{ for all }i\in\cU\}.\]
If $\cU$ does not satisfy~\ref{prop:single}, we can take $i\in\cU$ such that $\Phi(i,j)=0$ for all $j\in\cU$ and $\sqrt{\alpha_i}\notin K_i(\zeta_{c_i})$ so that $\chi_i\psi$ is non-trivial for any Dirichlet character $\psi\bmod c_i$.
Then the argument is similar to that for $\mathbf{A}\in\cF_4$, but where $b$ is instead taken from $(\Z/c_i\Z)^{\times}$ and the sum is over $p_{\ell}\equiv b\bmod c_i$.

If $\cU$ does not satisfy~\ref{prop:double}, take $\{i,j\}\subseteq\cU$ such that $\Phi(i,k)=\Phi(j,k)$ for all $k\in\cU$, assume that $\sqrt{\alpha_i\alpha_j}\notin K_i(\zeta_{c_i})\cdot  K_j(\zeta_{c_j})$.
Set
\[\varphi(D_i)=\prod_{k\in\cU\setminus\{i,j\}}\leg{D_i}{D_k}^{\Phi(i,k)}\leg{D_k}{D_i}^{\Phi(k,i)}.\]
Notice that $\Phi(i,j)=\Phi(j,j)=0$ and also $\Phi(j,i)=\Phi(i,i)=0$ by assumption.
Since the indices in $\mathcal{U}$ are unlinked,
it follows from \eqref{eq:quad-rec} that 
 $\varphi(D_i)$ is either $\leg{-1}{D_i}$ or trivial depending on $D_k$ and the values of $\Phi(i,k)$ for $k\in\mathcal{U}\setminus\{i,j\}$.
Since $\Phi(i,k)=\Phi(j,k)$ for all $k\in\cU$, we have
\[S(N,\mathbf{A})\ll
\sum_{(D_k)_{k\in\cU\setminus\{i,j\}}}
\left|\sum_{A_i\leq D_i<\Delta A_i}
f_i(D_i)\chi_i(D_i)\varphi(D_i)\sum_{A_j\leq D_j<\Delta A_j}f_j(D_j)\chi_j(D_j)\varphi(D_j)\right|.
\]
Since $\sqrt{\alpha_i\alpha_j}\notin K_i(\zeta_{c_i})\cdot  K_j(\zeta_{c_j})$, we deduce that 
$K_i(\sqrt{\alpha_i})\cdot K_j(\sqrt{\alpha_j})\not\subseteq K_i(\zeta_{c_i})\cdot  K_j(\zeta_{c_j})$ and $K_i(\sqrt{-\alpha_i})\cdot K_j(\sqrt{-\alpha_j})\not\subseteq K_i(\zeta_{c_i})\cdot  K_j(\zeta_{c_j})$.
Therefore at least one of 
$K_i(\sqrt{\alpha_i})\not\subseteq K_i(\zeta_{c_i})$ and $K_j(\sqrt{\alpha_j})\not\subseteq  K_j(\zeta_{c_j})$ holds, and similarly at least one of 
$K_i(\sqrt{-\alpha_i})\not\subseteq K_i(\zeta_{c_i})$ and $K_j(\sqrt{-\alpha_j})\not\subseteq  K_j(\zeta_{c_j})$ holds.
This allows us to take some $k\in\{i,j\}$ such that the character $\chi_k\varphi\psi$ is non-trivial for every Dirichlet character $\psi\bmod c_k$.
Then the argument proceeds similarly as in the case for $\mathbf{A}\in\cF_4$ by applying Lemma~\ref{lemma:SW} to $\chi_k\varphi\psi$.

Therefore we conclude that
\begin{equation}\label{eq:notadmissible}
 \sum_{\mathbf{A}\in\cF_5}|S(N,\mathbf{A})|
\ll N(\log N)^{-1}.
\end{equation}

\subsection{Proof of Theorem~\ref{theorem:charsum}} 
By Lemma~\ref{lemma:omegabound},~\eqref{eq:incomplete},\eqref{eq:fewlarge},~\eqref{eq:largelinked},~\eqref{eq:largesmalllinked}, and~\eqref{eq:notadmissible}, it remains to consider 
\[\sum_{\mathbf{A}\notin\cF_1\cup\cF_2\cup\cF_3\cup\cF_4\cup\cF_5}S(N,\mathbf{A}).\]
Since $\mathbf{A}\notin\cF_2$, there are at least $M$ indices such that $A_i\geq N^{\ddagger}$. Call this set of indices $\cU$.
Any two indices in $\cU$ must be unlinked because $\mathbf{A}\notin\cF_3$. The assumption 
that  any unlinked set must have size $\leq M$ implies that $\cU$ must have size exactly $M$.
Moreover any other indices not in $\cU$ must be linked to some indices in $\cU$, and so we are forced to have $A_j=1$ for any $j\notin\cU$ because $\mathbf{A}\notin\cF_3\cup\cF_4$. 

Now if $\cU\subseteq\cJ_k$ for some $k\in\{1,\dots,r\}$, then $A_j=1$ for all $j\in\cI\setminus\cJ_k$. This would contradict the condition $\prod_{i\in\cI\setminus\cJ_k}D_i\neq 1$. Therefore we can assume that $\cU\not\subseteq\cJ_k$.
Also $\mathbf{A}\notin\cF_5$ allows us to assume that $\cU$ is admissible.
Substituting $D_i=1$ for all $i\notin\cU$ gives the main term in Theorem~\ref{theorem:charsum}.

\section{Quadratic twists with full two-torsion}\label{s:7}

In this section, we will prove Theorem~\ref{theorem:mainfullt}.
The lower bound in Theorem~\ref{theorem:mainfullt} follows from Lemma~\ref{lemma:twistslb}, 
and so it remains to prove the upper bound.
For convenience, we shall work with a different model for the elliptic curves. If a monic polynomial $f\in \ZZ[x]$ has three distinct roots $r_1<r_2<r_3$ over $\Q$, then $r_1,r_2,r_3\in\Z$ and it follows that $f(x+r_1)=x(x-(r_2-r_1))(x-(r_3-r_1))$.
Therefore, on using  a linear transformation, it suffices to consider the integral points with respect to the model
\begin{equation}\label{eq:model1}
E_D:y^2=x(x-AD)(x-BD),
\end{equation}
where $0<A<B$ are integers that are fixed, and $D\in\D$.
Accordingly, we define
\[E_D^*(\Z)\coloneqq E_D(\Z)\setminus E_D[2]= \left\{(x,y)\in\Z^2: y\neq 0,\ y^2=x(x-AD)(x-BD)\right\}.\]
The following result is a refinement of 
Theorem~\ref{theorem:mainfullt}.

\begin{theorem}\label{theorem:fulltors}
Let $A,B\in\Z$ such that $0<A<B$ and consider the model \eqref{eq:model1}.
Then 
\[
\#\{D\in\D(N):E^*_D(\Z)\neq \varnothing\}
\ll 
\begin{cases}
N(\log N)^{-\frac{1}{4}+\epsilon}&\text{if }AB\notin\Q^2\text{ and }B(B-A)\notin\Q^2,\\
N(\log N)^{-\frac{1}{8}+\epsilon}&\text{if }AB\in\Q^2\text{ or }B(B-A)\in\Q^2,
\end{cases}\]
where the implied constant depends at most on $A,B$ and $\epsilon$.
\end{theorem}

Define 
\[\D^+\coloneqq \{D>0:D\in\Z\text{ square-free}\}\]
and 
\[\D^+(N)\coloneqq \{D\in\D^+:D \leq N\}.
\]
We may assume that $\gcd(A,B)=1$, since otherwise we can consider the larger family where $A$ and $B$ are replaced by $A/\gcd(A,B)$ and $B/\gcd(A,B)$. This assumption also implies that $A,B,B-A$ are pairwise coprime.
The main result of this section is the following proposition.

\begin{prop}\label{prop:fulltwolarge}
Let $A,B\in\Z$ such that $0<A<B$ and $\gcd(A,B)=1$, and consider the model \eqref{eq:model1}.
Fix $\kappa>0$. Then
\[\#\left\{D\in\D^+(N):
\begin{array}{l}
\gcd(x,D)\geq N(\log N)^{-\kappa}\\
\text{for some }(x,y)\in E_D^*(\Z)\end{array}\right\}
\ll
\begin{cases}
N(\log N)^{-\frac{1}{4}+\epsilon}&\text{if }AB\notin\Q^2\\
N(\log N)^{-\frac{1}{8}+\epsilon}&\text{if }AB\in\Q^2,
\end{cases}\]
where the implied constant depend at most on $A,B$, $\kappa$ and $\epsilon$.
\end{prop}

Suppose $(x,y)\in E_D^*(\Z)$, so that $y\neq 0$.
Let $g\coloneqq\gcd(x,D)$ and 
write $x=g\tilde{x}$ and $D=g\tilde{D}$. 
Then from the equation for $E_D$, we see that $g^3\mid y^2$. Since $D$ is square-free, $g$ must also be square-free, so $g^2\mid y$ and we write $y=g^2\tilde{y}$. We may now rewrite the equation as 
\begin{equation}\label{eq:factorgcd2tors}
g\tilde{y}^2=\tilde{x}(\tilde{x}-A\tilde{D})(\tilde{x}-B\tilde{D}).
\end{equation}
Here $\gcd(\tilde{x},\tilde{D})=1$, so 
\begin{align*}
 \gcd(\tilde{x},\tilde{x}-A\tilde{D})&\mid A\\
\gcd(\tilde{x},\tilde{x}-B\tilde{D})&\mid B\\
\gcd(\tilde{x}-A\tilde{D},\tilde{x}-B\tilde{D})& \mid B-A.
\end{align*}
We proceed by rewriting the factors appearing on the right hand side of~\eqref{eq:factorgcd2tors} as
\begin{equation}\label{eq:factorrhs}
\begin{split} \tilde{x}&=G_1y_1^2,\\
\tilde{x}-A\tilde{D}&=G_2y_2^2,\\
\tilde{x}-B\tilde{D}&=G_3y_3^2,
\end{split}
\end{equation}
where $G_1,G_2,G_3$ are square-free integers and $y_1,y_2,y_3$ are non-zero integers, such that $G_1G_2G_3=g(\delta_{1}\delta_{2}\delta_{3})^2$ with
\[ 
\delta_{1}\coloneqq\gcd(G_2,G_3)\mid B-A,\qquad 
\delta_{2}\coloneqq\gcd(G_3,G_1)\mid B,\qquad
 \delta_{3}\coloneqq\gcd(G_1,G_2)\mid A.
\]
Since $A,B,B-A$ are pairwise coprime, $\delta_{1},\delta_{2},\delta_{3}$ are also pairwise coprime.
If $p\mid \gcd(\delta_{i},\tilde{D})$, then $p\mid \tilde{x}$ by \eqref{eq:factorrhs}. But $\gcd(\tilde{x},\tilde{D})=1$, and so it must be that $\gcd(\delta_{1}\delta_2\delta_3,\tilde{D})=1$.
This observation, together with $\gcd(g,\tilde{D})=1$, implies that $\gcd(G_1G_2G_3,\tilde{D})=1$.
Taking the difference between the equations in~\eqref{eq:factorrhs}, the system can be rewritten as \begin{equation}\label{eq:conic3}
 \begin{split}
 G_1y_1^2-G_2y_2^2&=A\tilde{D}\\
G_1y_1^2-G_3y_3^2&=B\tilde{D}\\
G_2y_2^2-G_3y_3^2&=(B-A)\tilde{D}.
\end{split}
\end{equation}

We have seen that $(x,y)\in E_D^*(\Z)$ gives a system of the form \eqref{eq:conic3}. For given distinct $i$ and $j$, consider the map
\begin{equation}\label{eq:run}
\left\{(x,D):D\in\D,\ (x,y)\in E_D(\Z)\text{ for some }y\neq 0\right\}\rightarrow \Z_{\neq 0}\times\Z_{\neq 0},
\end{equation}
given by $(x,D)\mapsto (G_iy_i^2,G_jy_j^2)$.
Given $(G_iy_i^2,G_jy_j^2)$, we can recover 
the value of $\tilde{D}$ from  the equations in \eqref{eq:conic3}, and then the value of $G_ky_k^2$, where $k\not\in\{ i,j\}$. Since $y_k\neq 0$ and $G_k$ is square-free, this is enough to recover $G_k$.
Similarly, we can recover $G_i$ and $G_j$. We also get $g$ from the square-free part of $G_1G_2G_3$. Finally we obtain a pair $(x,D)$ through the identities  $x=g\tilde{x}=gG_1y_1^2$ and $D=g\tilde{D}$.

The $G_i$ we have constructed are square-free but not necessarily all positive. We now show the case with any negative $G_i$ has negligible contribution.
Over $\R$, the curve $E_D$ has two connected components, we bound the number of integral points that lie in the compact component.
\begin{lemma}\label{lemma:negativepoints}
Suppose $B>A$ are coprime positive integers.
Fix $\kappa>0$.
\[\sum_{D\in\D^+(N)}\#\left\{(x,y)\in E_D^*(\Z):x<BD\text{ and }\gcd(x,D)\geq \frac{N}{(\log N)^{\kappa}}\right\}\ll (\log N)^{2\kappa},\]
where the implied constant depends at most on $A,B$.\end{lemma}
\begin{proof}
If $(x,y)\in E_D^*(\Z)$ is such that $x<BD$, then we have a solution $(y_1,y_2,y_3)$ to~\eqref{eq:conic3} with $G_2,G_3<0$ and $G_1>0$. The condition $\gcd(x,D)\geq \frac{N}{(\log N)^{\kappa}}$ implies that $\tilde{D}\leq (\log N)^{\kappa}$. Looking at the second equation $G_1y_1^2-G_3y_3^2=B\tilde{D}$, the terms on the left are both positive, so $G_1y_1^2, G_3y_3^2\ll (\log N)^{\kappa}$. This gives $\ll(\log N)^{2\kappa}$ possible $(G_1y_1^2,G_3y_3^2)$. Finally we use the fact that the map \eqref{eq:run} has $O(1)$ fibres above 
$(G_1y_1^2,G_3y_3^2)$, giving $O(1)$ choices for $y,x,D$ overall.
\end{proof}

Lemma~\ref{lemma:negativepoints} allows us to restrict to the points with $G_1,G_2,G_3>0$ in the rest of our argument.

\subsection{$2$-Selmer elements from integral points with large $\gcd(x,D)$}
We will use Theorem~\ref{theorem:charsum} to prove the following lemma.

\begin{lemma}\label{lemma:positivepoints}
Suppose $B>A$ are coprime positive integers.
Fix $\kappa>0$. Then
\[\#\left\{D\in\D^+(N):
\begin{array}{l}x\geq BD\text{ and }\\
\gcd(x,D)\geq N(\log N)^{-\kappa}\\\text{for some }(x,y)\in E_D^*(\Z)\end{array}\right\}\ll 
\begin{cases}
N(\log N)^{-\frac{1}{4}+\epsilon}&\text{if }AB\notin\Q^2\\
N(\log N)^{-\frac{1}{8}+\epsilon}&\text{if }AB\in\Q^2,
\end{cases}\]
where the implied constant depend at most on $A,B$, $\kappa$ and $\epsilon$.
\end{lemma}

We
collect from~\eqref{eq:conic3} the local solvability conditions at the primes $p$ dividing $g$, but not $AB(B-A)\tilde{D}$. Note that such primes are necessarily odd since $AB(B-A)$ is even.  These conditions may be written 
\[
\begin{cases}
\leg{-A\tilde{D}G_2}{p}=\leg{-B\tilde{D}G_3}{p}=1&\text{if }p\mid G_1,\\
\leg{A\tilde{D}G_1}{p}=\leg{-(B-A)\tilde{D}G_3}{p}=1&\text{if }p\mid G_2,\\
\leg{B\tilde{D}G_1}{p}=\leg{(B-A)\tilde{D}G_2}{p}=1&\text{if }p\mid G_3.
\end{cases}
\]
For each $i$, define $\gamma_i=\gcd(AB(B-A),G_i)$ and $n_i$ such that
$$
G_i=n_i \gamma_i. 
$$
Thus  $\gcd(n_i,AB(B-A))=1$ and $n_i$ is square-free, since $G_i$ is square-free. 
Define
 \begin{align*}
 R_{13}&=-A\tilde{D}\cdot \gamma_2,
 &R_{12}&=-B\tilde{D}\cdot \gamma_3,
 &R_{21}&=-(B-A)\tilde{D}\cdot \gamma_3,\\
 R_{23}&=A\tilde{D}\cdot \gamma_1,
 &R_{32}&=B\tilde{D}\cdot \gamma_1,
 &R_{31}&=(B-A)\tilde{D}\cdot \gamma_2,
 \end{align*}
and 
$$
 R_{i0}=1, \quad  R_{i4}= \prod_{k\in\{1,2,3\}\setminus\{i\}}R_{ik},
$$
for $i\in\{1,2,3\}$.

We may now rewrite the local conditions at the primes dividing $n_1n_2n_3$ as
\begin{equation}\label{eq:forcharsum}
\begin{cases}
 \leg{R_{13}n_2}{p}=\leg{R_{12}n_3}{p}=1&\text{if }p\mid n_1,\\
 \leg{R_{23}n_1}{p}=\leg{R_{21}n_3}{p}=1&\text{if }p\mid n_2,\\
 \leg{R_{32}n_1}{p}=\leg{R_{31}n_2}{p}=1&\text{if }p\mid n_3.\\
\end{cases}
\end{equation}
Then, given $\gamma_1,\gamma_2,\gamma_3,\tilde{D}$,~the latter conditions are satisfied by $(n_1,n_2,n_3)$ if and only if the expression
\begin{multline} \label{eq:indictortriple}
\frac{1}{4^{\omega(n_1n_2n_3)}}\prod_{p\mid n_1}\left(1+\leg{R_{13}n_2}{p}\right)\left(1+\leg{R_{12}n_3}{p}\right)
\\
\times \prod_{p\mid n_2}\left(1+\leg{R_{23}n_1}{p}\right)
\left(1+\leg{R_{21}n_3}{p}\right)\prod_{p\mid n_3}\left(1+\leg{R_{32}n_1}{p}\right)\left(1+\leg{R_{31}n_2}{p}\right)
\end{multline}
is equal to $1$.
We can expand the first product as
\[ \prod_{p\mid n_1}\left(1+\leg{R_{13}n_2}{p}\right)\left(1+\leg{R_{12}n_3}{p}\right)
=
\sum_{n_1=n_{10}n_{12}n_{13}n_{14}}
\leg{R_{13}n_2}{n_{13}n_{14}}
\leg{R_{12}n_3}{n_{12}n_{14}}.
\]
Similarly we expand the other products with indexing
\[n_1=n_{10}n_{12}n_{13}n_{14},\quad n_2=n_{20}n_{21}n_{23}n_{24},\quad 
n_3=n_{30}n_{31}n_{32}n_{34}.\]
We will sum~\eqref{eq:indictortriple} over positive square free integers $n_1,n_2,n_3$ coprime to $AB(B-A)\tilde{D}$, with the exceptions
\begin{equation} \label{eq:exclusions}
\begin{cases}
n_1n_2=1&\text{if }\tilde{D}=1,\ 
AB\text{ not square},\ R_{32}\text{ and }R_{31}\text{ are squares,}\\
\hfil n_1=1&\text{if }\tilde{D}=1,\ AB,R_{32}\text{ are squares}.\end{cases}
\end{equation}

\begin{lemma}\label{lemma:genericgcd}
Let $\kappa>0$. Fix positive integers $\gamma_1,\gamma_2,\gamma_3\mid AB(B-A)$ and take a square-free positive integer $\tilde{D} \leq(\log N)^{\kappa}$ that is coprime to $\gamma_1\gamma_2\gamma_3$. Let $\tau$ be the square-free part of $\gamma_1\gamma_2\gamma_3\tilde{D}$.
We have
\[
\#\left\{D\in\D^+(N):\begin{array}{l}
 n_1,n_2,n_3\geq 1,\ n_1n_2n_3\tau=D\\
\eqref{eq:forcharsum}\text{ holds, but }
\eqref{eq:exclusions} \text{ is not satisfied}\\
 \gcd(n_1n_2n_3,AB(B-A))=1
\end{array}\right\}
\ll \frac{N}{\tilde{D}}(\log N)^{-\frac{1}{4}+\epsilon},
\]
where the implied constant depends at most on $A,B$, $\kappa$ and $\epsilon$.
\end{lemma}
\begin{proof}
Let $R(N)$ denote the quantity that is to be estimated. 
Summing~\eqref{eq:indictortriple} over the $(n_1,n_2,n_3)$ satisfying the assumptions, and then expanding the sum, we get
\begin{align*} 
R(N)&=\#\left\{(n_1,n_2,n_3)\in\Z^3_{>0}:\begin{array}{l}
 n_1n_2n_3\tau\in\D(N)\\
 \eqref{eq:forcharsum}\text{ holds, but }
\eqref{eq:exclusions} \text{ is not satisfied}\\
 \gcd(n_1n_2n_3,AB(B-A))=1
\end{array}\right\}\\
&=\sum_{(n_{ij})}
\prod_{i} \frac{1}{4^{\omega(n_{i0})}} 
\prod_{j\neq 0}
\frac{1}{4^{\omega(n_{ij})}}\leg{R_{ij}}{n_{ij}}
\prod_{k\not\in\{i,j\}}\prod_{0\leq l\leq 4} \leg{n_{kl}}{n_{ij}},
\end{align*}
where the sum is over tuples of positive integers $(n_{ij})$ such that $\prod_{ij} n_{ij}\in\D(\frac{N}{\tau})$, such that~\eqref{eq:exclusions} does not hold, but $\gcd(\prod n_{ij},AB(B-A) \tau)=1$.
We apply Theorem~\ref{theorem:charsum} with 
\[\cI=\{10,12,13,14,20,21,23,24,30,31,32,34\},\]
\[ \cJ=
\begin{cases}
\hfil \{30,31,32,34\}&\text{if }\tilde{D}=1,\ 
AB\text{ not square},\ R_{32}\text{ and }R_{31}\text{ are squares,}\\
 \{20,21,23,24,30,31,32,34\}&\text{if }\tilde{D}=1,\ AB,R_{32}\text{ are squares},\\
\hfil \varnothing &\text{otherwise}.
\end{cases}\]
Furthermore, we take $\lambda=\frac{1}{4}$ and 
\[\Phi(kl,ij)=\begin{cases}
 1&\text{if }k\not\in\{ i,j\}\text{ and }j\neq 0,\\
 0&\text{otherwise}.
\end{cases}\]
Finally, we take $f_i(n)=4^{-\omega(n)}$ and $c_i=1$  for all $i\in\cI$.
By \cite[Lemma~9]{HBSelmer1}, the maximal unlinked sets in $\cI$ are of size $4$, and those that are within $\cI$ are
\[\{i0,ij,ik,i4\},\ \{i0, j0, ij, ji\},\ \{ik, i4, jk, j4\},\]
where $i,j,k\in\{1,2,3\}$ denotes different non-zero indices.
Therefore the error term in Theorem~\ref{theorem:charsum} becomes $O(\frac{N}{\tilde{D}}(\log N)^{-\frac{1}{4}+\epsilon})$, which is satisfactory.
It remains to treat the main term in Theorem~\ref{theorem:charsum}. We claim that the only maximal unlinked sets that are admissible coincide with $\cJ$, so that the main term vanishes.

The first case to check is $\cU=\{i0,ij,ik,i4\}$. Here $\Phi(u,v)=0$ for all $u,v\in\cU$, so for~\ref{prop:single} to hold, $R_{ij}$ and $R_{ik}$ are both squares. Since $R_{12},R_{21},R_{13}<0$, the only possible case is when $i=3$, and $R_{31}=(B-A)\tilde{D}\cdot \gamma_2$ and $R_{32}=B\tilde{D}\cdot \gamma_1$ are squares. Since $\tilde{D}$ is coprime to $\gamma_1\gamma_2$ and square-free, it must be that $\tilde{D}\mid B$ and $\tilde{D}\mid B-A$. However by assumptions $B$ and $B-A$ are coprime, so $\tilde{D}=1$. This possibility therefore lies in $\cJ$.
 
In the second case $\cU=\{i0, j0, ij, ji\}$, and again $\Phi(u,v)=0$ for all $u,v\in\cU$. Thus~\ref{prop:single} implies that $R_{ij}$ and $R_{ji}$ are both squares. Since $R_{12},R_{21},R_{13}<0$, the only possible case is when $\{i,j\}=\{2,3\}$, and $R_{23}=A\tilde{D}\cdot \gamma_1$ and $R_{32}=B\tilde{D}\cdot \gamma_1
$ are squares. Then $R_{23}R_{32}$ is a square, and so $AB$ is a square. 
This means that $A$ and $B$ are squares, since $\gcd(A,B)=1$. 
By construction $\tilde{D}$ is square-free and coprime to $\gamma_1$, 
but $\tilde D \gamma_1$ is a square and 
so we conclude that $\tilde{D}=1$, which thereby  leads us to $\cJ$.

For the third case, $\cU=\{ik, i4, jk, j4\}$, and we can check that
\begin{align*}
\Phi(ik,i4)&=\Phi(jk,j4)=0\\
\Phi(ik,jk)&=\Phi(ik,j4)=\Phi(i4,jk)=\Phi(i4,j4)=1.\end{align*}
We check~\ref{prop:double} with $\{ik,i4\}$ or $\{jk,j4\}$.
Then $R_{ik}R_{i4}$ and $R_{jk}R_{j4}$ are both squares, so $R_{ij}$ and $R_{ji}$ are both squares. This again force us into $\cJ$, similarly to the  second case.
\end{proof}

\subsection{Exceptional points of the first kind}
We now treat the first of the cases that were previously excluded, as listed in \eqref{eq:exclusions}.
The following  result on simultaneous Pell equations \cite[Theorem~1.2]{Bugeaud} will prove crucial. 

\begin{lemma}\label{lemma:simPell}
Let $a,b$ be positive integers and $u,v$ be non-zero integers. Then the number of positive integer solutions $(x,y)$ to the equation $ax^2-by^2=u$, 
such that 
\[cy^2-dz^2=v\text{ for some positive integers }c,d,z\text{ satisfying }ab<cd \text{ and }cd\in\D(N),\]
is bounded by 
$O(\sqrt{N}(\log N)^2)$, 
where the implied constant depends at most on $a,b,u,v$.
\end{lemma}
\begin{proof}
Rescale the equations to
\begin{align*}
(ax)^2-aby^2&=au,\\
(dz)^2-cdy^2&=-dv.
\end{align*}
First suppose that none of $ab$, $cd$ and $abcd$ are squares.
Then by \cite[Theorem~1.2]{Bugeaud}, we have the upper bound
\begin{equation}\label{eq:simP}
\max\{ax,y,dz\}\leq \max\{|au|,|dv|,2\}^{C\sqrt{abcd}(\log ab)(\log cd)},\end{equation}
for  an absolute constant  $C>0$. 
Suppose we have a solution to the first equation and 
let $\varepsilon$ be the fundamental unit of $\Q(\sqrt{ab})$. Then there exists a positive integer $k$
such that 
\[\varepsilon^k\leq ax+\sqrt{ab}y<\varepsilon^{k+1}.\]
We refer to $\alpha=\varepsilon^{-k}(ax+\sqrt{ab}y)$  as the base solution, 
noting that $1\leq \alpha<\varepsilon$
and $|\overline{\alpha}|=|au/\alpha|\leq |au|$.
In particular, the number of base solutions only depends on $a,b,u$.

The solutions to the  first equation all take the form
\[ax+\sqrt{ab}y=\alpha\varepsilon^{l},\] 
for positive integers $l$. Thus, given  a base solution, it follows from  \eqref{eq:simP} that  
\[\varepsilon^{l}\ll \max\{|au|,|dv|,2\}^{C\sqrt{abcd}(\log ab)(\log cd)}.\]
Taking logs and noting that $cd\leq N$,
we have
\[l\ll \sqrt{abcd}(\log ab)(\log cd)(\log \max\{|au|,|dv|,2\})+1\ll \sqrt{N}(\log N)^2,
\]
which gives the claimed upper bound.

It remains to check the case when one of $ab$, $cd$ and $abcd$ is a square.
Since $cd$ is square-free and $cd>ab\geq 1$, $cd$ and $abcd$ cannot be squares. 
If $ab$ is a square, then we can factor the first equation as $(ax-\sqrt{ab}y)(ax+\sqrt{ab}y)=au$ over $\Q$. There are finitely many ways to factor $au$ as two factors over $\Q$, from which we can solve for $x$ and $y$. Therefore in this case the number of solutions $(x,y)$ is bounded in terms of $a$ and $u$.
\end{proof}

The first case in \eqref{eq:exclusions} comes from integral points of the form $(BD,0)+2E_D(\Q)$. In this case $\tilde{D}=n_1=n_2=1$, and $R_{32}=B\gamma_1$, $R_{31}=(B-A)\gamma_2$ are both squares, so $G_1B$ and $G_2(B-A)$ are squares. Putting back to \eqref{eq:factorrhs}
we see that $xBD$ and $(x-AD)(B-A)D$ are squares.
\begin{lemma}\label{lemma:Lsquare}
We have
\[\sum_{D\in\D^+(N)}
\#\left\{(x,y)\in E_D^*(\Z):\begin{array}{l}
 xBD\in\Q^2,\ (x-AD)(B-A)D\in\Q^2,\\ \gcd(x,D)=D,\ x>BD
\end{array}\right\}
\ll\sqrt{N}(\log N)^2,
\]
where the implied constant depends at most on $A,B$.
 \end{lemma}
 \begin{proof}
The conditions implies that $ \tilde{x}B$ and $( \tilde{x}-A)(B-A)$ are squares. In the notation of \eqref{eq:factorrhs}, $G_1$ is the square-free part of $B$ and $G_2$ is the square-free part of $B-A$.
Substitute
\[BG_1y_1^2=U^2\quad \text{ and }\quad (B-A)G_2y_2^2=V^2\]
into the first two equations of~\eqref{eq:conic3},
\[
 \begin{cases}
(B-A)U^2-BV^2=AB(B-A)\\
U^2-BG_3y_3^2=B^2.
\end{cases}
\]
By Lemma~\ref{lemma:simPell}, the number of positive integers $(U,V)$ satisfying the equations is bounded by $\ll_{A,B} \sqrt{N}(\log N)^2$. We can recover the integral point $(x,\pm y)$ from each $(U,V)$.
 \end{proof}

\subsection{Exceptional points of the second kind}
We proceed to treat the second case in \eqref{eq:exclusions}, and so we assume that $AB, R_{32}$ are squares, $\tilde{D}=1$ and $n_1=1$. In particular, $D=g$ and $G_1=\gamma_1$.
Since $A,B$ are coprime, we can write $A=a^2$, $B=b^2$ for some integers $a,b$. 
Since $R_{32}=b^2\gamma_1$ is a square, we see that $G_1$ is a square and hence $\tilde{x}$ is a square. Thus $x=g\tilde{x}=Du^2$ for some integer $u$. 
Finally the conditions $G_1,G_2,G_3>0$ are equivalent to  $x>BD$. In this way we are led to tackle the following result. 

\begin{lemma}\label{lemma:cyclic4}
Let $B>A$ be coprime positive integers that are both squares.
We have
\[\sum_{D\in\D^+(N)}\#\left\{(x,y)\in E_D^*(\Z):
xD\in\Q^2,\ x>BD
\right\}
\ll N(\log N)^{-\frac{1}{8}+\epsilon},
\]
where
the implied constant depends at most on $A,B$.
\end{lemma}
\begin{proof}
Suppose $(x,y)\in E_D^*(\Z)$ and $x=Du^2$, where $u$ is a positive integer. Substituting  $x=Du^2$ into the equation $y^2=x(x-a^2D)(x-b^2D)$, we obtain
\[(u^2-a^2)(u^2-b^2)=Dt^2,\]
for some non-zero integer $t$, since $y\neq 0$.
Then from the factorisation we can write
\begin{equation}\label{eq:Gy2}
 \begin{split}
 u-a&=g_1y_1^2,\\
 u+a&=g_2y_2^2,\\
 u-b&=g_3y_3^2,\\
	u+b&=g_4y_4^2.
\end{split}
\end{equation}
where $g_1,g_2,g_3,g_4$ are positive square-free integers such that $g_1g_2g_3g_4=Dv^2$ for some integer $v$. If $p^k\mid v$ then we must have $p^k\mid \gcd(g_i,g_j)$ for some $i\neq j$. In this way, since
$a,b$ are assumed to be coprime, it follows that $v\mid 2ab(b^2-a^2)$.
Write 
$$n_i=\frac{g_i}{\gcd(g_i,2ab(b^2-a^2))},
$$
for $1\leq i\leq 4$. 
We easily conclude that 
\begin{equation}\label{eq:charsum4}
\begin{cases}
\hfil  \leg{2ag_2}{p}=\leg{(b-a)g_3}{p}=\leg{(a+b)g_4}{p}=1 &\text{if }p\mid n_1,\\
 \hfil \leg{-2ag_1}{p}=\leg{-(a+b)g_3}{p}=\leg{(b-a)g_4}{p}=1&\text{if }p\mid n_2,\\
 \hfil \leg{-(b-a)g_1}{p}=\leg{(a+b)g_2}{p}=\leg{2bg_4}{p}=1&\text{if }p\mid n_3,\\
 \hfil \leg{-(a+b)g_1}{p}=\leg{-(b-a)g_2}{p}=\leg{-2bg_3}{p}=1&\text{if }p\mid n_4.
\end{cases}
\end{equation}
Note that $n_1\cdots n_4\mid g_1\cdots g_4=Dv^2$, whence
$n_1\cdots n_4\leq N$, since $n_1,\dots,n_4$ are coprime to $v$. 

In order to prove the lemma we will first focus on  bounding the 
quantity
\begin{align*} 
R(N)&=\#\left\{(n_1,n_2,n_3,n_4)\in\Z^3_{>0}:\begin{array}{l}
 n_1n_2n_3n_4\in\D(N),
 \eqref{eq:charsum4}\text{ holds}\\
 \gcd(n_1n_2n_3n_4,ab(b^2-a^2))=1\\
  n_in_j\neq 1\text{ for all }\{i,j\}\subset\{1,2,3,4\}
\end{array}\right\}.
\end{align*}
We shall deal with the case in which 
$ n_in_j=1$ for some $\{i,j\}\subset\{1,2,3,4\}$ at the end of the proof. 
Clearly,  there exist  $R_i\mid a^2b^2(b^2-a^2)^2$, for $1\leq i\leq 4$, such that 
\eqref{eq:charsum4} can be rewritten as
\[\leg{R_in_i}{p}=1\text{ if }p\mid n_j,\]
whenever $i\neq j$. 
Since there are only finitely many possible $R_1,R_2,R_3,R_4$ given $a$ and  $b$, we will view $R_1,R_2,R_3,R_4$ as fixed.
The condition on all $p\mid n_1$ can be packaged as
\begin{align*}
\Pi_1
&=\frac{1}{8^{\omega(n_1)}} \prod_{p\mid n_1}\left(1+\leg{R_{2}n_2}{p}\right)\left(1+\leg{R_{3}n_3}{p}\right)\left(1+\leg{R_{4}n_4}{p}\right)\\
&=\frac{1}{8^{\omega(n_1)}}
\sum_{n_1=\prod_{S\in\cP(\{2,3,4\})}n_{1,S}}
\leg{R_{2}n_2}{n_{1,\{2\}}}
\leg{R_{3}n_3}{n_{1,\{3\}}}
\leg{R_{4}n_4}{n_{1,\{4\}}}
\leg{R_{2}R_{3}n_2n_3}{n_{1,\{2,3\}}}\\
&\hspace{12em}
\times \leg{R_{2}R_{4}n_2n_4}{n_{1,\{2,4\}}}
\leg{R_{3}R_{4}n_3n_4}{n_{1,\{3,4\}}}
\leg{R_{2}R_{3}R_{4}n_2n_3n_4}{n_{1,\{2,3,4\}}},
\end{align*}
where $\cP$ denotes the power set. We define $\Pi_2,\Pi_3,\Pi_4$
similarly, and expand the products with indexing
\[n_1=\prod_{S\in\mathcal{P}(\{2,3,4\})}
\hspace{-0.2cm}
n_{1,S},\quad n_2=\prod_{S\in\mathcal{P}(\{1,3,4\})}
\hspace{-0.2cm}n_{2,S},\quad n_3=\prod_{S\in\mathcal{P}(\{1,2,4\})}
\hspace{-0.2cm}n_{3,S},\quad  n_4=\prod_{S\in\mathcal{P}(\{1,2,3\})}
\hspace{-0.2cm}n_{4,S}.\]
Define
\[\Phi((i,S),(j,S'))=\begin{cases}
1&\text{if }j\in S,\\
0&\text{otherwise},
\end{cases}\]
and take the set of indices to be
\[\cI=\left\{(i,S):i\in\{1,2,3,4\},\ S\in\cP\left(\{1,2,3,4\}\setminus\{i\}\right)\right\}.\]
Then we may write
\begin{align*} 
R(N)
&=\sum_{(n_{\u})}\Pi_1\Pi_2\Pi_3\Pi_4
=\sum_{(n_{\u})}
\prod_{\u\in\cI}
\frac{1}{8^{\omega(n_{\u})}}\leg{R_{\u}}{n_{\u}}
\prod_{\u,\v\in\cI} \leg{n_{\u}}{n_{\v}}^{\Phi(\u,\v)},
\end{align*}
where $R_\u$ depends on $R_1,R_2,R_3,R_4$, and the sum is over tuples of positive integers $(n_{\u})$ such that $\prod_{\u} n_{\u}\in\D(N)$ and 
$\gcd(\prod n_{\u },ab(b^2-a^2))=1$, and such that 
 $n_{\u}n_{\v}\neq 1$ for all distinct $\u,\v$.
We apply Theorem~\ref{theorem:charsum} with 
$\lambda=\frac{1}{8}$ and 
\[\cJ_{i,j}=\{(i,S)\in\cI\}\cup \{(j,S)\in\cI\}\text{ for }\{i,j\}\subset\{1,2,3,4\}.\]
Furthermore, we take 
$f_\u(n)=8^{-\omega(n)}$
and $c_\u=1$, 
for all $\u\in \mathcal{I}$.
We can check that the maximal unlinked sets have size $8$ and each of them is contained in  one of the $\cJ_{i,j}$. Thus the sum over $\cU$ vanishes. Indeed, suppose that $(i_1,S_1)$ is unlinked to $(i_2,S_2)$ and $(i_3,S_3)$, where $i_1,i_2,i_3$ are distinct. Then,  whether or not $S_1$ contains $i_2$ is determined by whether or not $i_1$ is contained in $S_2$. Similarly,  whether or not $S_1$ contains $i_3$ is determined by whether or not $i_1$ is in $S_3$. This only leaves $2$ possible choices for $S_1$,  given $S_2$ and $S_3$. Hence the size of any unlinked indices, not of the form $\cJ_{i,j}$, is bounded by $6$.
Therefore, on putting $M=8$, 
it follows fromTheorem~\ref{theorem:charsum}  that the sum is bounded by $O(N(\log N)^{-\frac{1}{8}+\epsilon})$.

It remains to deal with the remaining cases not considered in $R(N)$, in which two of $n_1,n_2,n_3,n_4$ are equal to $1$. We assume that $n_i=n_j=1$, with  $\{i,j,k,l\}=\{1,2,3,4\}$. Take the difference between the equations in \eqref{eq:Gy2} that are
associated to $n_i$ and $n_j$, and then the difference between those associated to $n_k$ and $n_k$, in order to get
\begin{align*}
g_iy_i^2-g_jy_j^2&=c_i-c_j,\\
g_ky_k^2-g_ly_l^2&=c_k-c_l,
\end{align*}
where $c_1=-a, c_2=a,c_3=-b, c_4=b$.
Notice that $g_i, g_j\mid ab(b^2-a^2)$ because $n_i=n_j=1$. 
Therefore we can apply Lemma~\ref{lemma:simPell} to get  the overall bound $O(\sqrt{N}(\log N)^2)$ in this case.
\end{proof}

\subsection{Conclusion}

\begin{proof}[Proof of Lemma~\ref{lemma:positivepoints}]
We apply Lemma~\ref{lemma:genericgcd}. The exceptions in \eqref{eq:exclusions} are dealt with in Lemma~\ref{lemma:Lsquare} and Lemma~\ref{lemma:cyclic4}.
 \end{proof}
\begin{proof}[Proof of Proposition~\ref{prop:fulltwolarge}]
Combine Lemma~\ref{lemma:negativepoints} and Lemma~\ref{lemma:positivepoints}. 
 \end{proof}

\begin{proof}[Proof of Theorem~\ref{theorem:fulltors}]
Apply Proposition~\ref{prop:fulltwolarge} with $\kappa=13$. 
For the points $(x,y)\in E^*_D(\Z)$ with small $\gcd(x,D)$, we observe that $(X,Y)=(9x-3(A+B)D,27y)\in\Z^2$ gives a non-trivial integral point on the short Weierstrass model 
\[Y^2=X^3 -27 (A^2 + B^2 -AB)D^2X-27 (2 B-A) (B-2A) (A + B)D^3.\]
Moreover, the relation $9x=X+3(A+B)D$ implies that $\gcd(X,D)\mid 9\gcd(x,D)$. We may therefore apply Lemma~\ref{lemma:smallgcd} with $K=9N(\log N)^{-\kappa}$ to  bound the number of integral points with $\gcd(x,D)\leq N(\log N)^{-\kappa}$. We conclude that
 \[
\#\left\{D\in\D^+(N):
\begin{array}{l}
\gcd(x,D)\leq N(\log N)^{-\kappa}\\\text{for some }(x,y)\in E_D^*(\Z)\end{array}\right\}
\ll N(\log N)^{-\frac{1}{2}\kappa+6},
\]
which  proves the desired upper bound for $D\in\D^+(N)$. 
 
 To treat negative $D\in\D(N)$, we instead consider  the family $$
 E_D:y^2=x(x-(B-A)D)(x-BD),
 $$ 
 with $D\in\D^+(N)$ Thus  the role of $A$ is replaced by $B-A$, and the argument runs as before, leading to the same conclusion. 
\end{proof}

\section{Quadratic twists with partial two-torsion}\label{s:8}
To prove Theorem~\ref{theorem:mainpartt}, we consider instead the model
\begin{equation}\label{eq:model2}
E_D:y^2=x(x^2+ADx+BD^2),
\end{equation}
for integers $A,B$ such that $A^2-4B\notin\Q^2$.
For such curves, the only two-torsion points are the point at infinity and $(0,0)$, so
\[
E_D^*(\Z)=E_D(\Z)\setminus\{(0,0)\}=\{(x,y)\in\Z^2: y^2=x(x^2+ADx+BD^2),\ y\neq 0\}.
\]
The lower bound in Theorem~\ref{theorem:mainpartt} follows from Lemma~\ref{lemma:twistslb}, whereas the  upper bound will follow from the next result. 

\begin{theorem}\label{theorem:partialtors}
Let $E_D$ be given by \eqref{eq:model2}, for  $A,B\in\Z$ such that $A^2-4B\notin\Q^2$. 
If $B<0$ or $B\in\Q^2$, then
\[
\#\left\{D\in\D(N): E_D^*(\Z)\neq \varnothing\right\}
\ll N(\log N)^{-\frac{1}{8}}\log\log N.
\]
If $B>0$ and $B\notin\Q^2$, then
\[
\#\left\{D\in\D(N):\begin{array}{l}
xB\notin\Q^2 \text{ or }
\gcd(x,D)<N(\log N)^{-\frac{49}{4}}\\
\text{or }x<\exp(N(\log N)^{-\frac{99}{8}})\\
\text{for some }(x,y)\in E_D^*(\Z)
\end{array}
\right\}
\ll N(\log N)^{-\frac{1}{8}}\log\log N.
\]
The implied constants depend at most on $A$ and $B$.
\end{theorem}

Conjecture \ref{con} pertains to elliptic curves in short Weierstrass form. After a suitable  change of variables, the equation \eqref{eq:model2} defining $E_D$ can be transformed into an equation of the form 
$y^2=x^3+A'x+B'$, with $A',B'\in \ZZ$  such that $A'\ll D^2$ and $B'\ll D^3$. But then it follows from 
Conjecture \ref{con} that there exists constants $C>0$ and $\ve>0$ such that any  
$(x,y)\in E_D^*(\Z)$ satisfies $x<\exp(C N^{1-\ve})$, 
if $D\in \D(N)$. That 
Theorem~\ref{theorem:mainpartt} follows from 
Theorem~\ref{theorem:partialtors}
 is now obvious.

Since Lemma~\ref{lemma:smallgcd} will be enough to deal with points with small $\gcd(x,D)$, we will focus our attention on the case when $\gcd(x,D)$ is large.
\begin{prop}\label{prop:partialgcd}
Let $E_D$ be given by \eqref{eq:model2}, for  $A,B\in\Z$ such that $A^2-4B\notin\Q^2$.
Let $\kappa>0$. Then 
\[
\#\left\{D\in\D^+(N):
\begin{array}{l}
\gcd(x,D)\geq N(\log N)^{-\kappa}\text{ and }xB\notin\Q^2\\
\text{for some }(x,y)\in E_D^*(\Z)\end{array}\right\}
\ll N(\log N)^{-\frac{1}{8}}\log\log N,
\]
where implied constant depends at most on $A$, $B$  and $\kappa$.
\end{prop}
Moving to the two-division field $K\coloneqq\Q(\sqrt{A^2-4B})$ of $E_D$, note that 
we can rewrite the equation as
\[y^2=x(x-\alpha D)(x-\overline{\alpha} D),\]
where 
\[\alpha\coloneqq\frac{-A+ \sqrt{A^2-4B}}{2}\quad \text{ and }\quad
\overline{\alpha}\coloneqq\frac{-A-\sqrt{A^2-4B}}{2}.\]
Over $\R$, the curve $E_D$ has two connected components if $A^2-4B>0$ and one if $A^2-4B<0$. 
If $A^2-4B<0$, the value of  $x^2+ADx+BD^2$ is always positive and it follows that also $x>0$. If 
 $A^2-4B>0$, on the other hand,  then we must have  $x>-\max\{|\alpha|,|\bar\alpha|\}D$ for any real point on $E_D$.

Suppose $(x,y)\in E_D^*(\Z)$ and let $g\coloneqq\gcd(x,D)$. We shall follow the opening steps in the proof of Proposition \ref{prop:fulltwolarge}.
Write $x=g\tilde{x}$ and $D=g\tilde{D}$, so that $y=g^2\tilde{y}$ for some integer $\tilde{y}$, as before.
Then substituting this  back into the equation, we obtain
\[g\tilde{y}^2=
\tilde{x}(\tilde{x}^2+A\tilde{D}\tilde{x}+B\tilde{D}^2)
=\tilde{x}(\tilde{x}-\alpha \tilde{D})(\tilde{x}-\overline{\alpha}\tilde{D}).\]
Observe that since $\gcd(\tilde{x},\tilde{D})=1$, we have
\[\delta\coloneqq\gcd(\tilde{x},\tilde{x}^2+A\tilde{D}\tilde{x}+B\tilde{D}^2)\mid B.
\]
Thus  we can factor $\tilde{x}$ and $\tilde{x}^2+A\tilde{D}\tilde{x}+B\tilde{D}^2$ over $\Q$ as
\begin{align}
\tilde{x}&=g_1\delta y_1^2\label{eq:partial1}
\\
(\tilde{x}-\alpha \tilde{D})(\tilde{x}-\overline{\alpha}\tilde{D})=\tilde{x}^2+A\tilde{D}\tilde{x}+B\tilde{D}^2&=g_2\delta y_2^2,\label{eq:partial2}
\end{align}
where $g_1,g_2, y_1,y_2$ are integers such that $g_1g_2=g$ and $y_1y_2=\tilde{y}$.
Substituting~\eqref{eq:partial1} into~\eqref{eq:partial2}, we obtain
\begin{equation}\label{eq:hades}
(g_1\delta y_1^2)^2+A\tilde{D}g_1\delta y_1^2+B\tilde{D}^2=g_2\delta y_2^2.
\end{equation}

In the next result,  which is an analogue of Lemma \ref{lemma:negativepoints}, 
we bound the number of integral points that lie in the compact component when there are two connected components.
\begin{lemma}\label{lemma:partialnegative}
Fix $\kappa>0$. Let $A,B\in \ZZ$ such that $A^2-4B\notin\Q^2$ and $A^2-4B>0$. Then
\[\sum_{D\in\D^+(N)}\#\left\{(x,y)\in E_D^*(\Z):
\begin{array}{l}
x<\max\{|\alpha|,|\overline{\alpha}|\}D\text{ and }\\
\gcd(x,D)\geq N(\log N)^{-\kappa}
\end{array}\right\}\ll (\log N)^{2\kappa},\]
where implied constant depends at most on $A,B$ and $\kappa$.
\end{lemma}
\begin{proof}
For such points $(x,y)\in E_D^*(\Z)$ with $D=\gcd(x,D)\tilde{D}$, we have  
$$
|\tilde{x}|<\max\{|\alpha|,|\overline{\alpha}|\}\tilde{D}\ll(\log N)^{\kappa}.
$$
There are $\ll(\log N)^{2\kappa}$ choices of $\tilde{x}$ and $\tilde{D}$, which can be used to recover square-free $g_1$ and $g_2$ by~\eqref{eq:partial1} and~\eqref{eq:partial2}, up to $\delta$, since $y_1,y_2\neq 0$ for non-torsion points.
\end{proof}

When $A^2-4B<0$, the value of  $x^2+ADx+BD^2$ is always positive. Thus  $g_2$ and hence also $g_1$ must be positive. 
If $A^2-4B>0$, on the other hand, it is possible that both $g_1,g_2$ are negative.  This happens precisely when $x=g\tilde{x}=g_2g_1^2\delta y_1^2$ is negative, but such points are handled by Lemma~\ref{lemma:partialnegative}. Thus, in what follows,  we can restrict our attention to the case in which   $g_1$ and $g_2$ are both positive.

We collect from \eqref{eq:hades} the local  solubility conditions at the primes  dividing $g_1$ and $g_2$, but not  $B(A^2-4B)$. These may be written
\[
\begin{cases}\hfil
\leg{B\delta g_2}{p}=1&\text{if }p\mid g_1\text{ and }p\nmid B,\\
\leg{A^2-4B}{p}=1&\text{if }p\mid g_2\text{ and }p\nmid A^2-4B.
\end{cases}
\]
In the above we have used that $p\mid g_1g_2$ implies $p\nmid \tilde{D}$, because $D=g_1g_2\tilde{D}$ is square-free.

For each $p\mid g_2$ and $p\nmid B(A^2-4B)$, the condition $\leg{A^2-4B}{p}=1$ implies that $p$ splits in $K/\Q$.
From~\eqref{eq:partial2}, we see that there is a prime $ \fp$ in $\O_K$ above $p$ such that
$\tilde{x}\equiv \alpha\tilde{D}\bmod  \fp$,
so substituting this into~\eqref{eq:partial1} yields
$g_1\delta y_1^2\equiv \alpha\tilde{D}\bmod  \fp$, whence
\[\leg{g_1\delta\alpha\tilde{D}}{\fp}=1.\]
Since $p$ splits in $K/\Q$, we have
\[\leg{g_1\delta\tilde{D}}{\fp}=\leg{g_1\delta\tilde{D}}{p}\]
and
\[\leg{B}{p}=\leg{B}{\fp}=\leg{\alpha\overline{\alpha}}{\fp}=\leg{g_1\delta\alpha \tilde{D}}{\fp}\leg{g_1\delta\overline{\alpha}\tilde{D}}{\fp}.\]
Therefore, if $\leg{B}{p}=-1$, then $\tilde{x}\equiv \alpha\tilde{D}\bmod \fp$ is automatically satisfied by one of the two primes above $p$.
If instead $\leg{B}{p}=1$, then it is clear that $\leg{\alpha}{\fp}$ does not depend on the choice of prime above $p$. 
Define $L\coloneqq K(\sqrt{\alpha})$. For any prime $p$ that splits completely in $K(\sqrt{B})/\Q$, we define
\[\leg{L/\Q}{p}\coloneqq \leg{\alpha}{\fp},\]
viewed as an Artin symbol in $\Gal(L/\Q)$, taking values in $\Gal(L/K(\sqrt{B}))\cong\F_2$.

Next, we write 
\[g_1=n_1\cdot \gamma_1\quad \text{ and }\quad g_2=n_2\cdot \gamma_2,\]
where $\gamma_1=\gcd(g_1,2B(A^2-4B))$ and $\gamma_2=\gcd(g_2,2B(A^2-4B))$.
Let 
\[
\mathcal{S}_{d}\coloneqq\left\{n\in\D^+:2\nmid n,~\leg{d}{p}=1 \text{ for all }p\mid n\right\}.
\] Also define
\[\omega_{d}(n)\coloneqq\#\left\{p\mid n:\leg{d}{p}=1\right\}.\]
Fix a choice of  $\delta\mid B$, $\gamma_1\gamma_2\mid 2B(A^2-4B)$, and $\tilde{D}$.
Define 
\[
R_{21}=\tilde{D}\delta \gamma_1\quad \text{ and }\quad R_{12}=B\delta \gamma_2.
\]
We would like  to count the number of pairs $(n_1,n_2)$ such that $n_1n_2\gamma_1\gamma_2\tilde{D}\in\D^+(N)$, with $\gcd(n_1n_2,2B(A^2-4B))=1$ and $n_2\in \mathcal{S}_{A^2-4B}$, and such that 
\begin{equation}\label{eq:localpartial}
\begin{cases}\hfil
\leg{R_{12}n_2}{p}=1&\text{if }p\mid n_1,\\
\leg{L/\Q}{p}\leg{R_{21}n_1}{p}=1 &\text{if }p\mid n_2\text{ and }\leg{B}{p}=1.\\
\end{cases}\end{equation}
For $n_1\in\D^+$ and $n_2\in\mathcal{S}_{A^2-4B}$, we note that the expression
\begin{equation}\label{eq:indicatorpartial}
\prod_{p\mid n_1}\frac{1}{2}\left(1+\leg{R_{12}n_2}{p}\right)\prod_{\substack{p\mid n_2\\ \leg{B}{p}=1}}\frac{1}{2}\left(1+\leg{L/\Q}{p}\leg{R_{21}n_1}{p}\right)
\end{equation}
is equal to $1$ if~\eqref{eq:localpartial} is satisfied by $n_1, n_2$, and $0$ otherwise.

\begin{lemma}\label{lemma:partialgenericgcd}
Let $\kappa>0$.
Fix positive integers $\delta\mid B$, $\gamma_1\gamma_2\mid 2B(A^2-4B)$. Take a square-free positive integer $\tilde{D}\leq (\log N)^{\kappa}$ that is coprime to $\gamma_1\gamma_2$.
We have
\[
\#\left\{D\in\D^+(N):\begin{array}{l}
 n_1,n_2\geq 1,\ n_1n_2\gamma_1\gamma_2\tilde{D}=D,\
 \eqref{eq:localpartial}\text{ holds} \\
 \gcd(n_1n_2,2B(A^2-4B))=1\\
 n_2\neq 1\text{ if }R_{12}\in\Q^2
 \end{array}\right\}
\ll \frac{N}{\tilde{D}}(\log N)^{-\frac{1}{8}},
\]
where implied constant depends at most on $A,B$ and $\kappa$.
\end{lemma}

\begin{proof}
Let $R(N)$ denote the quantity that is to be estimated. 
Summing~\eqref{eq:indicatorpartial} over $(n_1,n_2)$ such that $n_2\neq 1$ if $R_{12}\in\Q^2$, and then expanding the sum, we easily arrive at the expression
\begin{align*}
R(N)
&=
\sum_{\substack{n_{10},n_{12},n_{20},n_{21}\\n_{20}n_{21}\in\mathcal{S}_{A^2-4B}\\n_{21}\in\mathcal{S}_{B}}}\frac{1}{2^{\omega(n_{10}n_{12})+ \omega_B(n_{20}n_{21})}}\leg{R_{12}n_{20}n_{21}}{n_{12}}\leg{L/\Q}{n_{21}}\leg{R_{21}n_{10}n_{12}}{n_{21}},
\end{align*}
where the sum is over all positive integers $n_{10},n_{12},n_{20},n_{21}$ such that 
$$
n_{10}n_{12}n_{20}n_{21}\gamma_1\gamma_2\tilde{D}\in\D^+(N),
$$ 
with $\gcd(n_{10}n_{12}n_{20}n_{21},2B(A^2-4B))=1$, $n_{20}\in \mathcal{S}_{A^2-4B}$, $n_{21}\in\mathcal{S}_B\cap\mathcal{S}_{A^2-4B}$, and additionally $n_{20}n_{21}\neq 1$ if $R_{12}\in\Q^2$.

We plan to apply Theorem~\ref{theorem:charsum} to estimate $R(N)$. In order to facilitate this we  define
$$
R_{10}=R_{20}=1.
$$
Then, we shall  apply Theorem~\ref{theorem:charsum} with the following parameters. Firstly, we shall take 
\[\cI=\{10,12,20,21\},\quad \cJ=
\begin{cases}
\hfil \{10,12\}&\text{if }R_{12}\text{ is a square,}\\
\hfil \varnothing &\text{ otherwise},
\end{cases}\]
and 
\[\Phi(kl,ij)=\begin{cases}
 1&\text{if }k\neq i\text{ and }j\neq 0,\\
 0&\text{otherwise}.
\end{cases}
\]
The relevant arithmetic functions are 
\begin{align*}
f_{10}(n)=f_{12}(n)&=
\begin{cases}
2^{-\omega(n)} &  \text{ if $\gcd(n,2B(A^2-4B))=1$,}\\
0 &  \text{ otherwise,}
\end{cases}\\
f_{20}(n)&=
\begin{cases}
2^{-\omega_B(n)}&\text{if }n\in\mathcal{S}_{A^2-4B} \text{ and } \gcd(n,2B(A^2-4B))=1,\\
0&\text{otherwise,}
\end{cases}\\
f_{21}(n)&=\begin{cases}
2^{-\omega_B(n)}&\text{if }n\in\mathcal{S}_B\cap \mathcal{S}_{A^2-4B} \text{ and } \gcd(n,2B(A^2-4B))=1,\\
0&\text{otherwise.}
\end{cases}
\end{align*}
The relevant charaters are 
\begin{align*}
\chi_{10}(\cdot)=
\chi_{20}(\cdot)=1, \quad 
\chi_{12}(\cdot)=\left(\frac{R_{12}}{\cdot}\right), \quad 
\chi_{21}(\cdot)=\left(
\frac{L/\QQ}{\cdot}\right)
\left(\frac{R_{21}}{\cdot}\right),
\end{align*}
so that $K_{10}=K_{20}=K_{12}=\Q$, $K_{21}=K$, $\alpha_{10}=\alpha_{20}=1$, $\alpha_{12}=R_{12}$ and $\alpha_{21}=\alpha R_{21}$,
But then we may take 
$\lambda=\frac{1}{2}$ and 
$$
c_{10}=c_{12}=1,\quad c_{20}=c_{21}=2B(A^2-4B).
$$
We easily check that the  maximal unlinked sets are
\[\{10,12\},\{10,20\},\{12,21\},\{20,21\}.\]
Taking $M=2$, we see that the error term in Theorem~\ref{theorem:charsum} becomes $O(\frac{N}{\tilde{D}}(\log N)^{-\frac{1}{2}+\epsilon})$.

The sum in the main term in Theorem~\ref{theorem:charsum} is taken over all $\mathcal{U}$ satisfying \ref{prop:one}--\ref{prop:double}. Thus  we need to check which maximal unlinked sets are admissible (so that they satisfy \ref{prop:single} and \ref{prop:double}) and yet are not contained in $\mathcal{J}$.
Since $c_{12}=1$, it follows from \ref{prop:single} that 
the  set $\{10,12\}$ is only admissible if 
$R_{12}$ is a square. 
Therefore, the sum in the main term is only over $\cU\in\{\{10,20\},\{12,21\},\{20,21\}\}$. 
(Indeed, if $R_{12}$ is not a square then $\{10,12\}$ is not admissible and it shouldn’t appear in the main term;  alternatively, if $R_{12}$ is a square then $\{10,12\}$ is  admissible  but  contained in 
$\mathcal{J}$.)
Taking $\ve=\frac{5}{8}$, we may
 deduce that 
 $$
R(N)=M(N)+
O\left(\frac{N}{\tilde{D}}(\log N)^{-\frac{1}{8}}
\right),
$$
where
\begin{align*}
M(N)&\ll \sum_{\cU}\sum_{(D_i)_{i\in\cU}}
 \prod_{i\in\cU} f_i(D_i)\\
&\ll \sum_{\substack{n_{10},n_{20}\\ n_{10}n_{20}\leq N/\tilde{D}\\ n_{20}\in\mathcal{S}_{A^2-4B}}}\frac{1}{2^{\omega(n_{10})+\omega_B(n_{20})}}
+\hspace{-0.2cm}
\sum_{\substack{n_{12},n_{21}\\ n_{12}n_{21}\leq N/\tilde{D}\\ n_{21}\in\mathcal{S}_{A^2-4B}\cap\mathcal{S}_B}}
\hspace{-0.1cm}
\frac{1}{2^{\omega(n_{12})+\omega_B(n_{21})}}
+
\hspace{-0.2cm}
\sum_{\substack{n_{20},n_{21}\\ n_{20}n_{21}\leq N/\tilde{D}\\ n_{20}n_{21}\in\mathcal{S}_{A^2-4B}\\ n_{21}\in\mathcal{S}_B}}\hspace{-0.1cm}
\frac{1}{2^{\omega_B(n_{20}n_{21})}}.
\end{align*}
On appealing to Lemma \ref{lemma:Shiu}, we easily deduce that 
\begin{align*}
M(N)&\ll
 \frac{N}{\tilde{D}}\left((\log N)^{-\frac{1}{8}}+(\log N)^{-\frac{1}{4}}+(\log N)^{-\frac{3}{8}}\right)
\ll \frac{N}{\tilde{D}}(\log N)^{-\frac{1}{8}},
\end{align*}
which completes the proof of the lemma.
\end{proof}

Let us comment briefly on the  condition 
 $n_2\neq 1$  if $R_{12}\in\Q^2$, appearing in the counting function $R(N)$ in the proof of Lemma 
\ref{lemma:partialgenericgcd}. Dropping this condition would amount to taking  $\mathcal{J}$ to be the empty set in the proof. But then we would obtain a contribution from the admissible set $\{10,12\}$, which takes the shape
$
\sum_{n_{10}n_{12}\leq N/\tilde{D}} 2^{-\omega(n_{10}n_{12})}.
$
This sum has order $\frac{N}{\tilde{D}}\sqrt{\log N}$, which is much larger than the upper bound in 
Lemma~\ref{lemma:partialgenericgcd}.

\begin{proof}[Proof of Proposition~\ref{prop:partialgcd}]
Lemma~\ref{lemma:partialnegative} allows us to restrict to points $(x,y)\in E^*_D(\Z)$ that satisfy $x>0$. Next, we
sum the bound from Lemma~\ref{lemma:partialgenericgcd} over $\tilde{D} \leq(\log N)^{\kappa}$, and over all choices of $\delta, \gamma_1,\gamma_2$. 
The integral points $(x,y)\in E_D^*(\Z)$ that have not been handled satisfy $n_2= 1$ and $R_{12}\in\Q^2$, hence $Bx=Bgg_1\delta y_1^2=Bg_2\delta(g_1y_1)^2=R_{12}(g_1y_1)^2$ is a square. This proves Proposition~\ref{prop:partialgcd}.
\end{proof}

\subsection{Exceptional points}
To complete the proof of Theorem~\ref{theorem:partialtors}, it remains to deal with the set of exceptional points.
\begin{lemma}\label{lemma:p2excep}
Let $\kappa, \tau>0$. 
If $B<0$ or 
$B\in \QQ^2$, then 
\[
\sum_{D\in\D^+(N)}\#\left\{(x,y)\in E^*_D(\Z):
\gcd(x,D)\geq N(\log N)^{-\kappa}\text{ and }xB\in\Q^2\right\}
\ll (\log N)^{2\kappa}.
\]
If $B>0$ and $B\not\in \QQ^2$, then
\[
\sum_{D\in\D^+(N)}\#\left\{(x,y)\in E^*_D(\Z):
\begin{array}{l}
\gcd(x,D)\geq N(\log N)^{-\kappa}\\
xB\in\Q^2\\
x<\exp(N(\log N)^{-\kappa-\tau})
\end{array}
\right\}
\ll N(\log N)^{-\tau}(\log\log N)^{\frac{1}{2}}.
\]
The implied constants depend at most on $A,B$ and $\kappa$.
\end{lemma}
\begin{proof}
Lemma~\ref{lemma:partialnegative} allows us to restrict to points $(x,y)\in E^*_D(\Z)$ that satisfy $x>0$, so suppose that $g_1,g_2$ are positive integers in the notation of \eqref{eq:partial2}.
Since $Bx=Bgg_1\delta y_1^2=Bg_2\delta(g_1y_1)^2$ is a square, we deduce that $g_2$ is the squarefree part of $B/\delta$ and $B> 0$. 

Recall that $\gcd(\tilde{x},\tilde{D})=1$. Hence
the greatest common divisor of the ideals $(\tilde{x}-\alpha \tilde{D})$ and 
$(\tilde{x}-\overline{\alpha}\tilde{D})$ must divide 
$2\sqrt{A^2-4B}$.
But then, in the light of  \eqref{eq:partial2}, we may write 
\[\tilde{x}-\alpha \tilde{D}=\mathfrak{a}\mathfrak{b}^2,\]
for some ideals $\mathfrak{a}\mid 2B\sqrt{A^2-4B}$  and $\mathfrak{b}$ of $\O_K$.
Fix a representative for each ideal class of $\O_K$. Take $\mathfrak{c}$ to be the representative for the ideal class of $\mathfrak{b}$. Then $\mathfrak{b}/\mathfrak{c}$ must be a principal fractional ideal.
Therefore we may take $\mu\in \O_K$ and $\xi\in K$ such that $\mu\O_K=\mathfrak{a}\mathfrak{c}^2$ and $\xi\O_K=\mathfrak{b}\mathfrak{c}^{-1}$.  Note that, given $A$ and $B$, there are only finitely many possible $\mu$.
Then
\begin{equation}\label{eq:firstp}\tilde{x}-\alpha \tilde{D}=\mu\xi^2.\end{equation}
Taking conjugates in $K/\Q$, we obtain
\begin{equation}\label{eq:secondp}\tilde{x}-\overline{\alpha }\tilde{D}=\overline{\mu}\overline{\xi}^2.\end{equation}
Putting this back into \eqref{eq:partial2}, we see that
 $\overline{\mu}\mu(\overline{\xi}\xi)^2=g_2\delta y_2^2\in B\cdot \Q^2$. 
Therefore $\Q(\sqrt{\overline{\mu}\mu})=\Q(\sqrt{B})$.
Taking the difference of \eqref{eq:firstp} and \eqref{eq:secondp}, we have
\[\overline{\mu}\overline{\xi}^2-\mu\xi^2=\alpha \tilde{D}-\overline{\alpha }\tilde{D}=\tilde{D}\sqrt{A^2-4B}.\]
Squaring gives
\[(\overline{\mu}\overline{\xi}^2+\mu\xi^2)^2-4\overline{\mu}\mu(\overline{\xi}\xi)^2=\tilde{D}^2(A^2-4B).\]
Let $\eta=|\overline{\mu}\overline{\xi}^2+\mu\xi^2|+2\sqrt{\overline{\mu}\mu}|\overline{\xi}\xi|$ and $\eta'=|\overline{\mu}\overline{\xi}^2+\mu\xi^2|-2\sqrt{\overline{\mu}\mu}|\overline{\xi}\xi|$, which are both in $\O_{\Q(\sqrt{B})}$ since $\overline{\mu}\overline{\xi}^2,\mu\xi^2\in\O_K$ implies that $\overline{\mu}\overline{\xi}^2+\mu\xi^2\in\Z$ and $\overline{\mu}\mu(\overline{\xi}\xi)^2\in\Z$. Given $\tilde{D}$ and some $\delta\mid B$, each $(\eta,\eta')$ can only correspond to at most two pairs of $(x,D)$. Indeed, observe that $(\eta-\eta')^2=16\overline{\mu}\mu(\overline{\xi}\xi)^2=16g_2\delta y_2^2$, so \eqref{eq:partial2} allows us to recover $\tilde{x}$ as one of the two solutions to the quadratic equation, and hence $g_1$ from \eqref{eq:partial1}. Therefore it suffices to bound the number of $\eta$ that satisfy the  equation
\[\eta\eta'=\tilde{D}^2(A^2-4B)\]
over $\Q(\sqrt{B})$.

First suppose $B>0$ is a square and note that $\tau(d^2)\leq 3^{\omega(d)}$, for any $d\in \mathcal{D}^+$. We deduce that $\eta\in\Z$ divides $(A-4B^2)\tilde{D}^2$, so there are $O( 3^{\omega(\tilde{D})})$ choices for  $\eta$. Summing over all $\tilde{D}\leq (\log N)^{\kappa}$, the contribution is bounded by $\ll (\log N)^{2\kappa}$, by Lemma \ref{lemma:Shiu}.

Suppose next that $B>0$ is not a square. In this case $\eta'$ is the conjugate of $\eta$ in $\Q(\sqrt{B})$. 
There are $O (3^{\omega(\tilde{D})})$  choices of ideal $\mathfrak{d}=\eta\O_{\Q(\sqrt{B})}$ with norm $\tilde{D}^2|A^2-4B|$. Fixing the smallest generator $\beta$ of $\mathfrak{d}$ such that $\beta>1$, we see that $\eta$ must be of the form $\beta\epsilon^k$, where $k\geq 0$ is an integer and $\epsilon$ denotes the fundamental unit of $\Q(\sqrt{B})$.
By assumption $1<\eta\ll x+\tilde{D}\ll \exp(N(\log N)^{-\kappa-\tau})$. Thus  there are $O(N(\log N)^{-\kappa-\tau})$ possible $\eta$ given each $\mathfrak{d}$. We claim that any $p\mid \tilde{D}$ satisfies $\leg{B}{p}=1$ or $p\mid B$.
Indeed, if $p\mid \tilde{D}$ and $p\nmid B$,  then it follows from  \eqref{eq:partial2} that 
$g_2\delta$ is a square modulo $p$, which in turn implies that $B$ is a square modulo  $p$, since 
the first paragraph of the proof ensures that $Bg_2\delta$ is a square.
Putting everything together, the total contribution is found to be 
\[\ll N(\log N)^{-\kappa-\tau}\sum_{\substack{\tilde{D}\leq (\log N)^{\kappa}\\p\mid\tilde{D}\Rightarrow \leg{B}{p}\neq -1}}3^{\omega(\tilde{D})}\ll
N(\log N)^{-\tau}(\log\log N)^{\frac{1}{2}},
\]
by Lemma \ref{lemma:Shiu}.
\end{proof}

\subsection{Conclusion}

\begin{proof}[Proof of Theorem~\ref{theorem:partialtors}]
We apply Proposition~\ref{prop:partialgcd} and Lemma~\ref{lemma:p2excep} with $\kappa=\frac{49}{4}$ and $\tau=\frac{1}{8}$. We are then left with the points with $\gcd(x,D)<N(\log N)^{-\kappa}$, which we will handle with Lemma~\ref{lemma:smallgcd}.
Transforming the integral points $(x,y)\in E^*_D(\Z)$ to the integral points $(X,Y)=(9x+3AD,27y)\in\Z^2$ on the short Weierstrass model 
$$
Y^2=X^3+ 27 (3B   -  A^2) x+27A(2 A^2 - 9B),
$$ 
allows us to apply Lemma~\ref{lemma:smallgcd} with $K=9N(\log N)^{-\kappa}$. It follows that
\[
\#\left\{D\in\D^+(N):
\begin{array}{l}
\gcd(x,D)<N(\log N)^{-\kappa}\\\text{for some }(x,y)\in E_D^*(\Z)\end{array}\right\}
\ll N(\log N)^{-\frac{1}{2}\kappa+6}.
\]
This shows that the contribution from those $D\in\D^+(N)$ fits into the upper bound.

For the contribution from $D\in\D^-(N)$, we simply replace $A$ by $-A$ and consider instead $E_D:y^2=x(x^2-ADx+BD^2)$ with $D\in\D^+(N)$.
 \end{proof}


\begin{thebibliography}{10}

\bibitem{Alpoge}
L.~Alp\"oge,
\newblock The average number of integral points on elliptic curves is bounded.
{\em Preprint}, 2014.
 \newblock {\tt arXiv:1412.1047}.

\bibitem{AlpogeHo}
L.~Alp\"oge and W.~Ho,
\newblock The second moment of the number of integral points on elliptic curves
 is bounded.
\newblock 
{\em Preprint}, 2022.
\newblock {\tt arXiv:1807.03761}.

\bibitem{Bean}
M.~A. Bean,
\newblock An isoperimetric inequality for the area of plane regions defined by
 binary forms.
\newblock {\em Compositio Math.} {\bf 92} (1994), 115--131.

\bibitem{d4}
T.~D. Browning,
\newblock The density of rational points on a certain singular cubic surface.
\newblock {\em J.\ Number Theory} {\bf 119} (2006), 242--283.

\bibitem{Bugeaud}
Y.~Bugeaud,
\newblock Effective simultaneous rational approximation to pairs of real
 quadratic numbers.
\newblock {\em Mosc. J. Comb. Number Theory} {\bf 9} (2020),  
353--360.

\bibitem{ACL}
A. Chambert-Loir and Y. Tschinkel, Igusa integrals and volume asymptotics in analytic and adelic geometry. {\em Confluentes Math.} {\bf 2} (2010), 351--429. 


\bibitem{congruentav}
S.~Chan,
\newblock The average number of integral points on the congruent number curves.
{\em Preprint}, 2023.
\newblock {\tt arXiv:2112.01615}.

\bibitem{cubicav}
S.~Chan,
\newblock Integral points on cubic twists of {M}ordell curves.
\newblock {\em Math.\ Ann.}, to appear.

\bibitem{Cremona}
J.~E. Cremona,
\newblock Reduction of binary cubic and quartic forms.
\newblock {\em LMS J.\ Comput.\ Math.} {\bf 2} (1999), 64--94.

\bibitem{FK4rank}
\'{E}.~Fouvry and J.~Kl\"{u}ners,
\newblock On the 4-rank of class groups of quadratic number fields.
\newblock {\em Invent. Math.} {\bf 167} (2007), 455--513.

\bibitem{fmt}
{J. Franke, Y.~I. Manin and Y. Tschinkel},
\newblock Rational points of bounded height on Fano varieties.
\newblock {\em Invent.\ Math.} {\bf 95} (1989), 421--435.

\bibitem{Goldstein}
L.~J. {Goldstein,}
\newblock A generalization of the {S}iegel--{W}alfisz theorem.
\newblock {\em Trans. Amer. Math. Soc.} {\bf 149} (1970), 417--429.

\bibitem{Granvilletwists}
A.~Granville,
\newblock Rational and integral points on quadratic twists of a given
 hyperelliptic curve.
\newblock {\em Int. Math. Res. Not. IMRN} (2007), no.~8, Art.~ID 027, 24 pp.

\bibitem{Bug}
L. Hajdu and T. Herendi,
Explicit bounds for the solutions of elliptic
equations with rational coefficients.
{\em J. Symbolic Computation} {\bf 25} (1998), 361--366.

\bibitem{HR}
G.~H. Hardy and S.~Ramanujan,
The normal number of prime factors of a number {$n$}.
{\em {Q}uart. {J}.
 {M}ath.} {\bf 48} (1917), 76--92.
 


\bibitem{HBcharsums}
D.~R. Heath-Brown,
\newblock A mean value estimate for real character sums.
\newblock {\em Acta Arith.} {\bf 72} (1995),  235--275.

\bibitem{HBlattice}
D.~R. Heath-Brown,
\newblock Diophantine approximation with square-free numbers.
\newblock {\em Math. Z.} (1984) {\bf 187}, 335--344.


\bibitem{HBSelmer1}
D.~R. Heath-Brown,
\newblock The size of {S}elmer groups for the congruent number problem.
\newblock {\em Invent. Math.} {\bf 111} (1993), 171--195.

\bibitem{HBSelmer}
D.~R. Heath-Brown,
\newblock The size of {S}elmer groups for the congruent number problem. {II}.
\newblock {\em Invent. Math.} {\bf 118} (1994), 331--370.

\bibitem{HS}
M.~Hindry and J.~H. Silverman,
\newblock The canonical height and integral points on elliptic curves.
\newblock {\em Invent. Math.} {\bf 93} (1988), 419--450.

\bibitem{Hux}
M.~N. Huxley,
\newblock A note on polynomial congruences.
\newblock In {\em Recent progress in analytic number theory, {V}ol.~1
 ({D}urham, 1979)}, pp 193--196, Academic Press, London--New York, 1981.


\bibitem{lang}
S. Lang, 
Conjectured Diophantine estimates on elliptic curves.
{\em Arithmetic and geometry, Vol. I}, 155--171, 
Progr. Math. {\bf 35},
Birkh\"auser Boston, Boston, MA, 1983.

\bibitem{d4-a}
P.~Le~Boudec,
\newblock Affine congruences and rational points on a certain cubic surface.
\newblock {\em Algebra \& Number Theory} {\bf 8} (2014), 1259--1296.

\bibitem{Mordell}
L.~J. Mordell,
\newblock {\em Diophantine equations}.
\newblock Pure and Applied Mathematics {\bf 30}, Academic Press, London--New
 York, 1969.


\bibitem{schmidt}
W. Schmidt, 
Thue's equation over function fields.
{\em J.\ Austral.\ Math.\ Soc.} {\bf 25} (1978), 385--422.

\bibitem{SerreNp}
J.-P. Serre,
\newblock {\em Lectures on {$N_X (p)$}}. CRC
 Research Notes in Mathematics {\bf 11}, CRC Press, Boca Raton, FL, 2012.


\bibitem{Shiu}
P.~Shiu,
\newblock A {B}run-{T}itchmarsh theorem for multiplicative functions.
\newblock {\em J. reine angew. Math.} {\bf 313} (1980), 161--170.

\bibitem{Smith2}
A.~Smith,
\newblock The distribution of $\ell^{\infty}$-Selmer groups in degree $\ell$ twist families II.
\newblock 
{\em Preprint}, 2023.
\newblock {\tt arXiv:2207.05143}.

\bibitem{Thunderdecomp}
J.~L. Thunder,
\newblock Decomposable form inequalities.
\newblock {\em Ann. of Math.} {\bf 153} (2001), 767--804.

\bibitem{Wirsing}
E.~Wirsing,
\newblock Das asymptotische {V}erhalten von {S}ummen \"{u}ber multiplikative
 {F}unktionen.
\newblock {\em Math. Ann.} {\bf 143} (1961), 75--102.

\bibitem{Xiaopfree}
S.~Y. Xiao,
\newblock Power-free values of binary forms and the global determinant method.
\newblock {\em Int. Math. Res. Not. IMRN} {\bf 16} (2017), 5078--5135.


\end{thebibliography}
\end{document}